\documentclass[a4paper,10pt]{amsart}
\usepackage[utf8]{inputenc}
\usepackage{amsthm}
\usepackage{amsmath}
\usepackage{amsaddr}
\usepackage{bm}
\usepackage{amsfonts}
\usepackage{amssymb}
\usepackage{mathtools}
\usepackage{mathrsfs}
\usepackage{setspace}
\usepackage{xcolor}
\usepackage{textgreek}
\usepackage{todonotes}
\usepackage{esint}
\usepackage{thmtools} % solves problem of shared counters in statements and autoref

\usepackage[pdfdisplaydoctitle,colorlinks,breaklinks,urlcolor=blue,linkcolor=blue,citecolor=blue]{hyperref} %must always be the last

\newcommand{\N}{\mathbb{N}}
\newcommand{\R}{\mathbb{R}}
\newcommand{\PP}{\mathbb{P}}

\newcommand{\su}{\mathfrak{s}\mathfrak{u}}
\newcommand{\tildeE}[1]{\tilde{\mathbb{E}}\left[#1\right]}

\newcommand{\kk}{\textbf{k}}

\newcommand{\nn}{\textbf{n}}

\newcommand{\Cc}{\mathbb{C}}
\newcommand{\Ee}{\mathbb{E}}

\newcommand{\Rr}{\mathbb{R}}

\newcommand{\Ss}{\mathbb{S}}
\newcommand{\Tt}{\mathbb{T}}
\newcommand{\Zz}{\mathbb{Z}}

\newcommand{\U}{\mathrm{U}}

\DeclareMathOperator{\Tr}{Tr}

\def\XXint#1#2#3{{\setbox0=\hbox{$#1{#2#3}{\int}$ }
\vcenter{\hbox{$#2#3$ }}\kern-.6\wd0}}

\numberwithin{equation}{section}

\providecommand{\U}[1]{\protect\rule{.1in}{.1in}}
\newtheorem{theorem}{Theorem}

\newtheorem{corollary}[theorem]{Corollary}

\newtheorem{definition}[theorem]{Definition}

\newtheorem{lemma}[theorem]{Lemma}

\newtheorem{proposition}[theorem]{Proposition}
\newtheorem{remark}[theorem]{Remark}

\title[Zeitlin's 2D Euler equations]{On the infinite dimension limit of invariant measures and solutions of Zeitlin's 2D Euler equations}
\author[F. Flandoli]{Franco Flandoli}
  \address{Scuola Normale Superiore, Piazza dei Cavalieri, 7, 56126 Pisa, Italy}
\author[U. Pappalettera]{Umberto Pappalettera}
  \address{Scuola Normale Superiore, Piazza dei Cavalieri, 7, 56126 Pisa, Italy}
\author[M. Viviani]{Milo Viviani$^*$}
  \address{CRM Ennio De Giorgi, Scuola Normale Superiore, Piazza dei Cavalieri, 3, Pisa, 56126, Italy}
  \email{$^{\ast}$corresponding author \href{mailto:milo.viviani(at)sns.it}{milo.viviani(at)sns.it}}

\begin{document}

\begin{abstract}
In this work we consider a finite dimensional approximation for the 2D Euler equations on the sphere, proposed by V. Zeitlin, and show their convergence towards a solution to Euler equations with marginals distributed as the enstrophy measure. The method relies on nontrivial computations on the structure constants of $\mathbb{S}^2$, that appear to be new. In the last section we discuss the problem of extending our results to Gibbsian measures associated with higher Casimirs. 
\end{abstract}

\maketitle

%%%%%%%%%%%%%%%%%%%%%%%%%%%%%%%%%%%%%%%%%%%%%%%%%%%%%%%%%%%%%%%%%%%%%%%%%%%%%%%%%%%%%%%%%%%%%
\section{Introduction}
The 2D Euler equations are a fundamental mathematical model for studying ideal fluids, i.e. incompressible, inviscid, deformable bodies in which the dependence on one spatial dimension can be neglected. 
In particular, for barotropic incompressible fluids on some surface $S$ embedded in the Euclidean space $\Rr^3$, the Euler equations take the simple form
\begin{equation}\label{eq:Euler_eqn1}
\begin{array}{ll}
&\dot{\omega}=\nabla\psi^\perp\cdot\nabla\omega\\
&\Delta \psi=\omega,
\end{array}
\end{equation}
where $\omega$ and $\psi$ are respectively the vorticity and the stream function, $\nabla$ and $\Delta$ are respectively the Riemannian gradient and the Laplace--Beltrami operator on $S$.
One of the most intriguing aspects of these equations is the fact that they possess an infinite amount of conserved quantities, i.e. the integrals
\[
\int_S \psi\omega d\mbox{vol}_S, \hspace{1cm} \int_S f(\omega)d \mbox{vol}_S,
\]
where the first one represents the total kinetic energy and the second one the Casimir functions, defined for any $f\in C^1(\Rr)$.
These conservation laws are crucial in understanding the long-time behaviour of the fluid.
Indeed, as Kraichnan showed in \cite{Kr1967}, the conservation of both energy and enstrophy (i.e. the Casimir for $f(x)=x^2$) is responsible for the remarkable phenomenon of the formation and the persistence of large coherent vortices.

A first attempt to understand the statistical properties of 2D ideal fluids is due to Lars Onsager \cite{On1949}, who showed that in the simplified point-vortex model the equilibrium statistical mechanics predicts the concentration of vortices with the same sign.
More recently, the theory of Miller, Robert and Sommeria \cite{Mi1990,RoSo1991} extended the ideas of Onsager and Kraichnan taking into account all the invariants.
Via a mean field approach considering formally defined an invariant microcanonical measure for the Euler equations, they derive a functional relationship between the equilibrium average vorticity and the stream function \cite{BoVe2012}.
Even though the MRS theory has been quite recognized, several critical aspects and discrepancies with respect to experiments and numerical simulations have been found \cite{DrQiMa2015,MoVi2020}.
From a mathematical point of view, the MRS theory is purely formal and does not give any precise definition of the invariant measures considered.

At the moment only energy and enstrophy invariant measures have been rigorously constructed, and extending the existing results to other Casmirs is still an open problem.
Albeverio and Cruzeiro in \cite{AlCr1990} showed the existence of solutions to the 2D Euler equations as stochastic processes limit of Galerkin approximation of the Euler equations with vorticity in $H^{-\alpha}$, such that the enstrophy and the (renormalized) energy Gibbs measures are invariant for the flow.

In this paper, we consider a different finite dimensional approximation for the 2D Euler equations, valid on any orientable compact surface. 
This model was derived by V. Zeitlin \cite{ze1,ze2}, based on the theory of geometric quantization of compact Kähler manifolds \cite{BoMeSc1994}.
One of the main feature of Zeitlin's finite dimensional model is to posses a number of conserved quantities, which is proportional to the level of discretization and such that, for a sufficiently regular vorticity field, they approximate the original Casimirs of the 2D Euler equations.
In particular, for any level of discretization, the Zeitlin's model admits energy and enstrophy analogue, which are simply a spectral truncation of the original ones.

The aim of this work is to set a new theoretical framework in which developing a rigorous statistical theory for the Euler equations.
Indeed, one of the main open problems is defining Gibbsian invariant measures which takes into account other conserved quantities than energy and esntrophy.
Since these measures have distributional support, it is not clear, even up to renormalization, how to deal with higher order Casimirs of the Euler equations.
In this paper, we show that it is possible starting from the Zetlin's model to recover the results of Albeverio and Cruzeiro in \cite{AlCr1990}, but also that the Zeitlin’s model gives new insights in the problem, that in the future could allow to deal also with the other Casimirs.
Furthermore, a main novelty of our work is that we perform explicit calculations on the structure constants for the 2-sphere $\Ss^2$ (cfr. \autoref{sec:AppA}), which are technically more involved than those on the flat 2-torus (recalled in \autoref{sec:AppB} for completeness).

The paper is structured as follows.
In \autoref{sec:def} we present the geometric background necessary to set up the quantized version of Euler equations on $\mathbb{S}^2$: we introduce isometries between subspaces of functions on $\mathbb{S}^2$ and spaces of matrices in the Lie algebra $\su(N)$, as well as suitable Sobolev norms on $\su(N)$.
In \autoref{sec:gaussian_measures} we rigorously define a sequence of Gaussian measure on $\su(N)$ whose pull-back converges weakly towards the enstrophy measure, and prove useful bounds on stationary solutions of quantized Euler equations.
In \autoref{sec:limit} we show the existence of a subsequence of solutions of quantized Euler equations converging towards a limiting process $\tilde{\omega}$ taking values in a space of distributions: as a consequence of previous results, we are able to prove that $\tilde{\omega}$ is a stationary process with marginals distributed as the enstrophy measure, and that it solves a symmetrized version of Euler equations on $\mathbb{S}^2$.
Finally, in \autoref{sec:open} we discuss open problems, in particular concerning the difficulties encountered in trying to solve Euler equations having as invariant measure a Gibbsian measure associated to higher-order Casimirs, and we point out a tentative approach involving the evaluation of line integrals and Kelvin Theorem. 

\section{Fundamental concepts and definitions} \label{sec:def}
In this section, we introduce the fundamental concepts and notations that we employ throughout the paper.
In particular, in order to introduce the Zeitlin's model, we observe that the right hand side in the first equation of \eqref{eq:Euler_eqn1} defines a Poisson bracket denoted by:
\begin{equation}\label{eq:poisson_brack_sph}
\lbrace\psi,\omega\rbrace:=\nabla\psi^\perp\cdot\nabla\omega.
\end{equation}
The Poisson bracket notation highlights the infinite dimensional \textit{Lie--Poisson structure} of the Euler equations.
The main idea of Zeitlin's model is to define a finite dimensional approximation of the Euler equations, which retains the Lie--Poisson structure of the equations.
The functional space of vorticities is replaced, for any $N\geq 2$ by the Lie algebra $\su(N)$, defined as the tangent at the identity of $SU(N)$, which is the real vector space of dimension $d_N:=N^2-1$ of skew-Hermitian matrices with zero trace and Lie brackets $[W,V] \coloneqq WV-VW$, for $V,W\in\su(N)$.
The Laplace--Beltrami operator is replaced by a linear operator $\Delta_N$ defined on $\su(N)$, with the same spectrum (up to truncation) of $\Delta$.
In this paper, we perform our calculations on the 2-sphere $\Ss^2$ embedded in the Euclidean space $\Rr^3$ (in \autoref{sec:AppB} we show that the same results can be derived for the Zeitlin's model on the 2D flat torus).

The Zeitlin's model relies on the theory of \textit{geometric quantization} of the Poisson algebra $(C^\infty(\mathbb{S}^2),\{\cdot,\cdot\})$, \cite{BoMeSc1994}.
Let $Y_{\ell,m} \in C^\infty(\mathbb{S}^2)$ and $T^N_{\ell,m} \in \su(N)$ denote respectively the standard spherical harmonics and spherical matrices defined in \cite{HopYau1998}, for $\ell \in \mathbb{N}$, $m \in \mathbb{Z}$, $|m| \leq \ell$.
For this domain, the relationship between the functions and matrices is explicitly given in terms of spherical harmonics and spherical matrices.
Let us define the linear projectors: $\Pi_N : C^\infty(\mathbb{S}^2) \to \su(N)$, $N \in \mathbb{N}$, $N \geq 2$ satisfying:
\begin{itemize}
\item
for every $f,g \in C^\infty(\mathbb{S}^2)$, if $\|\Pi_N f - \Pi_N g \|_{\su(N)} \to 0$ as $N \to \infty$ then $f=g$;
\item
for every $f,g \in C^\infty(\mathbb{S}^2)$, $\Pi_N \{f,g\} = N^{3/2} [ \Pi_N f, \Pi_N g ] + O(1/N)$;
\item
$\Pi_N Y_{\ell,m} = T^N_{\ell,m}$, $\ell = 1,\dots,N-1$, $|m| \leq \ell$ is a basis of $\su(N)$.
\end{itemize}
Let us denote $L^2_N(\mathbb{S}^2) \coloneqq Span \left\{ Y_{\ell,m} , \,\ell = 1,\dots,N-1, |m| \leq \ell \right\}$, immersed in $C^\infty(\mathbb{S}^2)$ with immersion $\iota_N$.
The restriction of $\Pi_N$ to $L^2_N(\mathbb{S}^2)$ is isometric: for every $\ell,\ell' = 1,\dots,N-1$, $|m| \leq \ell$, $|m'| \leq \ell'$
\begin{align} \label{eq:isometry}
\delta_{\ell,\ell'} \delta_{m,m'}
=
\langle Y_{\ell,m},Y_{\ell',m'} \rangle_{L^2(\mathbb{S}^2)}
=
\langle T^N_{\ell,m},T^N_{\ell',m'} \rangle_{\su(N)}
\coloneqq
Tr((T^N_{\ell,m})^* T^N_{\ell',m'}).
\end{align} 
In the following, we denote $\tilde{j}_N : \su(N) \to L^2_N(\mathbb{S}^2)$ the inverse of the restriction of $\Pi_N$ to $L^2_N(\mathbb{S}^2)$, and $j_N = \iota_N \circ \tilde{j}_N : \su(N) \to C^\infty(\mathbb{S}^2)$. It is easy to check that $\Pi_N \circ j_N = Id_{\su(N)}$ and $j_N \circ \Pi_N$ is the orthogonal projector from $C^\infty(\mathbb{S}^2)$ onto $L^2_N(\mathbb{S}^2)$.

The discrete Laplacian $\Delta_N:\su(N) \to \su(N)$ acts on the basis $T^N_{\ell,m}$ as
\begin{align*}
\Delta_N  T^N_{\ell,m} = -\ell(\ell+1) T^N_{\ell,m}.
\end{align*}
Since also $\Delta Y_{\ell,m} = -\ell (\ell+1) Y_{\ell,m}$, we deduce for every $s \in \R$
\begin{align*}
\Pi_N (-\Delta)^s = (-\Delta_N)^s \Pi_N, 
\end{align*}
and thus we can define for $\omega = j_N W$
\begin{align} \label{eq:sobolev_norm}
\| W \|_{H^s(\su(N))} 
&\coloneqq 
\| \omega \|_{H^s(\mathbb{S}^2)}
=
\| (-\Delta)^{s/2} \omega \|_{L^2(\mathbb{S}^2)}
\\
&=
\| \Pi_N(-\Delta)^{s/2} \omega \|_{\su(N)}
=
\| (-\Delta_N)^{s/2} W \|_{\su(N)}, \nonumber
\end{align}
that is a good Sobolev norm on $\su(N)$, in the sense that Aubin-Lions and Simon compactness criterions hold \cite{Si1986}. 

The quantized Euler equations can be written as \cite{MoVi2020}:
\begin{align} \label{eq:euler}
\dot{W} = [P,W]_N=N^{3/2}[P,W] , \quad \Delta_N P = W. 
\end{align}
%\todo[inline]{la notazione $[P,W]_N$ non è stata introdotta precedentemente. Per evitare complicazioni, io partirei direttamente da $\dot{W} = N^{3/2}[P,W]$, citando uno dei lavori di Milo (o altri) a supporto di questa scelta.}
These equations have as conserved quantities the energy
\[
H(W) \coloneqq -1/2 \Tr(P^*W),
\]
the linear momentum
\[
M \coloneqq (W_{1,1},W_{1,0},W_{1,-1}),
\quad
W_{\ell,m} \coloneqq \langle W, Y_{\ell,m} \rangle_{{\su(N)}},
\]
and the Casimirs
\begin{align*}
C_k(W) \coloneqq Tr(W^k),
\quad k \in \N.
\end{align*}
Notice that for $k=2$ it holds $C_2(W) = - \|W\|^2_{\su(N)}$.

\section{Gaussian measures}\label{sec:gaussian_measures}
In this section we introduce the Gaussian measure on $\su(N)$ that permits us to prove the existence of stationary solutions to quantized Euler equations \eqref{eq:euler}. 
For this purpose, let $Q_N:\su(N) \to \su(N)$ be the covariance operator defined as 
\begin{align*}
\langle Q_N W , W' \rangle
\coloneqq
\frac1{2Z^{d_N}}
\int_{\su(N)} 
\langle \tilde{W} , W \rangle
\langle W' , \tilde{W} \rangle
e^{-\frac12 \|\tilde{W}\|^2} d\tilde{W},
\quad
W,W' \in \su(N),
\end{align*}
where $Z=\int_\Cc e^{-\frac12 |x|^2} dx$ is a suitable renormalization constant, and $d_N = N^2-1$. 
The covariance operator $Q_N$ is just a convenient rewriting of the identity operator on $\su(N)$, which is the content of the following:
\begin{lemma}
It holds $Q_N = Id_{\su(N)}$.
\end{lemma}
\begin{proof}
For notational convenience, let us relabel the basis $(T^N_{\ell,m})_{\ell=1,\dots,N-1, |m|\leq \ell}$ as $(T^N_k)_{k=1,\dots,d_N}$. 
Let $W = \sum_{k=1}^{d_N} c_k T^N_k$, $W' = \sum_{k=1}^{d_N} c'_k T^N_k$ and $\tilde{W} = \sum_{k=1}^{d_N} \tilde{c}_k T^N_k$.
We have
\begin{align*}
\langle Q_N W , W' \rangle
&=
\frac{1}{2Z^{d_N}}
\int_{\su(N)} 
\langle W , \tilde{W} \rangle
\langle \tilde{W}, W' \rangle
e^{-\frac12 \|\tilde{W}\|^2} d\tilde{W}
\\
&=
\frac{1}{2Z^{d_N}}
\int_{\Cc^{d_N}}
\sum_{k=1}^{d_N}
\overline{c_k} \tilde{c}_k
\sum_{h=1}^{d_N}
\overline{\tilde{c}_h} c'_h
\prod_{j=1}^{d_N}
e^{-\frac12 |\tilde{c}_j|^2} d\tilde{c}_j
\\
&=
\frac{1}{2Z^{d_N}}
\sum_{k,h=1}^{d_N}\int_{\Cc^{d_N}}
\overline{c_k} \tilde{c}_k
\overline{\tilde{c}_h} c'_h
\prod_{j=1}^{d_N}
e^{-\frac12 |\tilde{c}_j|^2} d\tilde{c}_j.
\end{align*}
Let us rearrange the product inside the integral in the following way. Denote $\{k,h\}$ the set with elements $k$ and $h$, and let $card\{k,h\}$ be its cardinality, so that $card\{k,h\}=1$ if $k=h$ and $card\{k,h\}=2$ if $k\neq h$.
Since in the previous expression the integration with respect to $d\tilde{c}_j$ produces only a factor $Z$ for $j \neq k,h$, we can rewrite 
\begin{align*}
\langle Q_N W , W' \rangle
&=
\frac{1}{2Z^{card\{k,h\}}}
\sum_{k,h=1}^{d_N}\int_{\Cc^{card\{k,h\}}} 
\overline{c_k} \tilde{c}_k
\overline{\tilde{c}_h} c'_h
\prod_{j \in \{k,h\}}
e^{-\frac12 |\tilde{c}_j|^2} d\tilde{c}_j
\\
&= 
\frac{1}{2Z}
\sum_{k=1}^{d_N} \int_{\Cc} 
\overline{c_k} \tilde{c}_k
\overline{\tilde{c}_k} c'_k
e^{-\frac12 |\tilde{c}_k|^2} d\tilde{c}_k
\\
&\quad+
\frac{1}{2Z^2}
\left(
\sum_{k=1}^{d_N}
\int_{\Cc} 
\overline{c_k} \tilde{c}_k
e^{-\frac12 |\tilde{c}_k|^2} d\tilde{c}_k
\right)
\left(
\sum_{h=1}^{d_N}
\int_{\Cc} 
\overline{\tilde{c}_h} c'_h
e^{-\frac12 |\tilde{c}_h|^2} d\tilde{c}_h
\right)
\\
&= \sum_{k} \overline{c_k} c'_k = \langle W,W' \rangle,
\end{align*}
where we deduce the last line from $\int_{\Cc} \tilde{c}_k e^{-\frac12 |\tilde{c}_k|^2} d\tilde{c}_k = 0$ and $\int_{\Cc} \tilde{c}_k \overline{\tilde{c}_k} e^{-\frac12 |\tilde{c}_k|^2} d\tilde{c}_k = 2Z$.
\end{proof}

\begin{corollary} \label{cor:WN}
Fix $\omega \in C^\infty(\mathbb{S}^2)$, and denote $W^{(N)}_\omega \coloneqq \Pi_N \omega$. Then
\begin{align*}
\lim_{N \to \infty} \langle Q_N W^{(N)}_\omega,W^{(N)}_{\omega'} \rangle
=
\int_{\mathbb{S}^2} \overline{\omega(x)}{\omega'(x)} d\mbox{vol}_S.
\end{align*}
\end{corollary}
\begin{proof}
It follows immediately from \eqref{eq:isometry} and the identity $\langle Q_N W^{(N)}_\omega,W^{(N)}_{\omega'} \rangle=\langle W^{(N)}_\omega,W^{(N)}_{\omega'} \rangle$, given by the previous lemma. 
\end{proof}

Denote $\mu_N(dW) \coloneqq \frac1{Z^{d_N}} e^{-\frac12 \|W\|_{\su(N)}^2} dW$ the Gaussian measure on $\su(N)$ with covariance $Q_N$, and let $\nu_N$ be its pull-back on $C^\infty(\mathbb{S}^2)$ given by $\nu_N \coloneqq (j_N)_* \mu_N$. The covariance of $\nu_N$ is given by $\tilde{Q}_N = j_N \circ \Pi_N$ (the orthogonal projector from $C^\infty(\mathbb{S}^2)$ to $L^2_N(\mathbb{S}^2)$); equivalently, the reproducing kernel of $\nu_N$ is $L^2_N(\mathbb{S}^2)$.

The \emph{enstrophy measure} is defined as the centered Gaussian measure $\nu$ on $H^{-1-}(\mathbb{S}^2) \coloneqq \cap_{s>0} H^{-1-s}(\mathbb{S}^2)$ with covariance $Q=Id$, or equivalently with reproducing kernel $L^2(\mathbb{S}^2)$.
The previous corollary implies $\nu_N \rightharpoonup \nu$ as measures on $H^{-1-}(\mathbb{S}^2)$.

We can now state the main results of this section.
\begin{lemma}
For every $\epsilon>0$ and $p \in [1,\infty)$ there exists a finite constant $C_{\epsilon,p}$ such that
\begin{align*}
\int_{\su(N)} \| W \|^p_{H^{-1-\epsilon}(\su(N))} \mu_N(dW)
\leq
C_{\epsilon,p}.
\end{align*}
\end{lemma}
\begin{proof}
Let $\omega = j_N W$. By \eqref{eq:sobolev_norm} and $\nu_N = (j_N)_* \mu_N$, change of variables yields
\begin{align*}
\int_{\su(N)} \| W \|^p_{H^{-1-\epsilon}(\su(N))} \mu_N(dW)
=
\int_{C^\infty(\mathbb{S}^2)} \| \omega \|^p_{H^{-1-\epsilon}(\mathbb{S}^2)} \nu_N(d \omega).
\end{align*}
For the measure $\nu_N$ the desired bound is classical, see for instance \cite[Section 3]{AlFe2004}.
\end{proof}

\begin{corollary}\label{cor:bounds}
Let $W^N_{W_0}:\Omega_N \times \R \to \su(N)$ be the solution of \eqref{eq:euler} with initial condition $W_0$ distributed as $\mu_N$.
For fixed $T>0$ denote $\hat{W}^N_{W_0}:\Omega_N \times [0,T] \to \su(N)$ the accelerated process
\begin{align*}
\hat{W}^N_{W_0} (t) = W^N_{W_0} (N^{3/2} t),
\quad
t \in [0,T]. 
\end{align*} 
Then for every $\epsilon>0$, $p \in [1,\infty)$ and $\kappa$ sufficiently large there exists a finite constant $C_{\epsilon,p,\kappa}$ such that  
\begin{align*}
\sup_{N \in \N} \mathbb{E}^{\mu_N}
\left[
\int_0^T
\| \hat{W}^N_{W_0}(t) \|^p_{H^{-1-\epsilon}(\su(N))} dt
+
\int_0^T
\| \frac{d}{dt} \hat{W}^N_{W_0}(t) \|^2_{H^{-\kappa}(\su(N))} dt
\right]
\leq
T C_{\epsilon,p,\kappa}.
\end{align*}

Similarly, let $\omega^N_{\omega_0}:\Omega_N \times [0,T] \to C^\infty(\mathbb{S}^2)$ be given by $\omega^{\omega_0} = j_N \hat{W}^N_{W_0}$. Then 
\begin{align*}
\sup_{N \in \N} \mathbb{E}^{\nu_N}
\left[
\int_0^T
\| \omega^N_{\omega_0}(t) \|^p_{H^{-1-\epsilon}(\mathbb{S}^2)} dt
+
\int_0^T
\| \frac{d}{dt} \omega^N_{\omega_0}(t) \|^2_{H^{-\kappa}(\mathbb{S}^2)} dt
\right]
\leq
T C_{\epsilon,p,\kappa}.
\end{align*}
\end{corollary}
\begin{proof}
First of all, notice that there exists a unique stationary solution to \eqref{eq:euler} by a suitable adaptation of non-explosion results in \cite[Section 3]{Cr1983}.
The dynamics of $\hat{W}^N_{W_0}$ is given by
\begin{align*}
\dot{\hat{W}}^N_{W_0} = N^{3/2} [P^N,\hat{W}^N_{W_0}],
\qquad \Delta_N P^N = \hat{W}^N_{W_0}.
\end{align*}
Let us introduce the streamfunction $\psi^N \coloneqq -(-\Delta)^{-1} \omega^N_{\omega_0}$.
It holds
\begin{align*}
\Pi_N \psi^N 
=
- \Pi_N (-\Delta)^{-1} \omega^N_{\omega_0}
=
- (-\Delta_N)^{-1} \Pi_N \omega^N_{\omega_0}
=
- (-\Delta_N)^{-1} \hat{W}^N_{W_0}
=
P^N,
\end{align*}
and therefore the dynamics of $\omega^N_{\omega_0}$ is given by
\begin{align} \label{eq:omega_dyn}
\dot{\omega}^N_{\omega_0}
&=
j_N \dot{\hat{W}}^N_{W_0}
=
j_N N^{3/2} [\Pi_N \psi^N,\Pi_N \omega^N_{\omega_0}]
=
j_N \Pi_N \{ \psi^N, \omega^N_{\omega_0} \} + j_N r^N,
\end{align}
with $r^N:\Omega_N \times [0,T] \to \su(N)$ given by 
\begin{align*}
r^N 
&= 
N^{3/2} [\Pi_N \psi^N,\Pi_N \omega^N_{\omega_0}] - \Pi_N \{ \psi^N, \omega^N_{\omega_0} \}.
\end{align*}
Writing $\omega^N_{\omega_0} \eqqcolon \sum_{\substack{\ell = 1,\dots,N-1,\\|m| \leq \ell}} \hat{\omega}_{\ell,m} Y_{\ell,m}$, by the previous formula we deduce
\begin{align} \label{eq:rN}
r^N 
&=
\sum_{\substack{\ell,\ell' = 1,\dots,N-1,\\|m| \leq \ell, |m'| \leq \ell'}}
\frac{\hat{\omega}_{\ell,m} \hat{\omega}_{\ell',m'}}{\ell(\ell+1)}
\left(
- N^{3/2} [ T^N_{\ell,m} , T^N_{\ell',m'} ] 
+ \Pi_N \{ Y_{\ell,m} , Y_{\ell',m'} \} 
\right)
\\
&\eqqcolon \nonumber
\sum_{\substack{\ell,\ell' = 1,\dots,N-1,\\|m| \leq \ell, |m'| \leq \ell'}}
\frac{\hat{\omega}_{\ell,m} \hat{\omega}_{\ell',m'}}{\ell(\ell+1)}
c^N_{\ell,\ell',m,m'}
\end{align}
with $\lim_{N \to \infty} \| c^N_{\ell,\ell',m,m'} \|_{\su(N)} = 0$ for every fixed $\ell,\ell',m,m'$ by the properties of $\Pi_N$.

Having said that, by \eqref{eq:sobolev_norm} and change of variables
\begin{gather*}
\mathbb{E}^{\mu_N}
\left[
\int_0^T
\| \hat{W}^N_{W_0}(t) \|^p_{H^{-1-\epsilon}(\su(N))} dt
+
\int_0^T
\| \frac{d}{dt} \hat{W}^N_{W_0}(t) \|^2_{H^{-\kappa}(\su(N))} dt
\right]
\\
=
\mathbb{E}^{\nu_N}
\left[
\int_0^T
\| \omega^N_{\omega_0}(t) \|^p_{H^{-1-\epsilon}(\mathbb{S}^2)} dt
+
\int_0^T
\| \frac{d}{dt} \omega^N_{\omega_0}(t) \|^2_{H^{-\kappa}(\mathbb{S}^2)} dt
\right].
\end{gather*}

Let us consider the two terms separately. The first one is easy to control, indeed
\begin{align*}
\mathbb{E}^{\nu_N}
\left[
\int_0^T
\| \omega^N_{\omega_0}(t) \|^p_{H^{-1-\epsilon}(\mathbb{S}^2)} dt
\right]
&=
\mathbb{E}^{\nu_N}
\left[
\int_0^T
\| \omega_0 \|^p_{H^{-1-\epsilon}(\mathbb{S}^2)} dt
\right]
= T C_{\epsilon,p};
\end{align*}
as for the second one, since $r^N$ is stationary as well
\begin{align*}
\mathbb{E}^{\nu_N}
&\left[
\int_0^T
\| \frac{d}{dt} \omega^N_{\omega_0}(t) \|^2_{H^{-\kappa}(\mathbb{S}^2)} dt
\right]
\\
&\qquad
\leq 
2\mathbb{E}^{\nu_N}
\left[
\int_0^T
\| \{ \psi^N(t), \omega^N_{\omega_0}(t) \} \|^2_{H^{-\kappa}(\mathbb{S}^2)} dt
+
\int_0^T
\| j_N r^N(t) \|^2_{H^{-\kappa}(\mathbb{S}^2)} dt
\right]
\\
&\qquad
= 
2T \mathbb{E}^{\nu_N}
\left[
\| \{ -(-\Delta)^{-1} \omega_0, \omega_0 \} \|^2_{H^{-\kappa}(\mathbb{S}^2)}
+
\| j_N r^N(0) \|^2_{H^{-\kappa}(\mathbb{S}^2)}
\right]
\\
&\qquad
\leq 
T C_{\epsilon,\kappa} + T 
\mathbb{E}^{\mu_N}
\left[
\| r^N(0) \|^2_{H^{-\kappa}(\su(N))}
\right].
\end{align*}

Writing 
\begin{align*}
c^N_{\ell,\ell',m,m'} 
\eqqcolon 
\sum_{\substack{\underline{\ell} = 1,\dots,N-1,\\|\underline{m}| \leq \underline{\ell}}} 
c^{N,\ell,\ell',m,m'}_{\underline{\ell},\underline{m}} T^N_{\underline{\ell},\underline{m}},
\end{align*}
we get
\begin{align*}
&\mathbb{E}^{\mu_N}
\left[
\| r^N(0) \|^2_{H^{-\kappa}(\su(N))}
\right]
\\
&\quad
=
\mathbb{E}^{\mu_N} 
\left[
\sum_{\substack{\underline{\ell} = 1,\dots,N-1,\\|\underline{m}| \leq \underline{\ell}}} 
(\underline{\ell}(\underline{\ell}+1))^{-\kappa}
\left|
\sum_{\substack{\ell,\ell' = 1,\dots,N-1,\\|m| \leq \ell, |m'| \leq \ell'}}
\frac{\hat{\omega}_{\ell,m} \hat{\omega}_{\ell',m'}}{\ell(\ell+1)}
c^{N,\ell,\ell',m,m'}_{\underline{\ell},\underline{m}}
\right|^2
\right]
\\
&\quad
=
\mathbb{E}^{\mu_N} 
\left[
\sum_{\underline{\ell},\underline{m}} 
(\underline{\ell}(\underline{\ell}+1))^{-\kappa}
\sum_{\substack{\ell,\ell',h,h',\\m,m',n,n'}}
\frac{\hat{\omega}_{\ell,m} \hat{\omega}_{\ell',m'}}{\ell(\ell+1)}
\frac{\overline{\hat{\omega}_{h,n}} 
\overline{\hat{\omega}_{h',n'}}}{h(h+1)}
c^{N,\ell,\ell',m,m'}_{\underline{\ell},\underline{m}}
\overline{c^{N,h,h',n,n'}_{\underline{\ell},\underline{m}}}
\right]
\\
&\quad
=
\sum_{\underline{\ell},\underline{m}} 
(\underline{\ell}(\underline{\ell}+1))^{-\kappa}
\sum_{\substack{\ell,\ell',h,h',\\m,m',n,n'}}
\frac{\mathbb{E}^{\mu_N} 
\left[\hat{\omega}_{\ell,m} \hat{\omega}_{\ell',m'} \overline{\hat{\omega}_{h,n}} 
\overline{\hat{\omega}_{h',n'}} \right]}
{\ell(\ell+1)h(h+1)}
c^{N,\ell,\ell',m,m'}_{\underline{\ell},\underline{m}}
\overline{c^{N,h,h',n,n'}_{\underline{\ell},\underline{m}}}.
\end{align*}
It holds that $\overline{\hat{\omega}}_{lm}=(-1)^m\hat{\omega}_{l-m}$.
Hence, by the Isserlis-Wick formula
\begin{align*}
\mathbb{E}^{\mu_N} 
\left[\hat{\omega}_{\ell,m} \hat{\omega}_{\ell',m'} \overline{\hat{\omega}_{h,n}} 
\overline{\hat{\omega}_{h',n'}} \right]
&=
(-1)^m(-1)^n\delta_{\ell,\ell'} \delta_{m,-m'} 
\delta_{h,h'} \delta_{n,-n'}
\\
&\quad
+
\delta_{\ell,h} \delta_{m,n} 
\delta_{\ell',h'} \delta_{m',n'}
\\
&\quad
+
\delta_{\ell,h'} \delta_{m,n'} 
\delta_{\ell',h} \delta_{m',n},
\end{align*}
and therefore
\begin{align*}
\mathbb{E}^{\mu_N}
\left[
\| r^N(0) \|^2_{H^{-\kappa}(\su(N))}
\right]
&=
\sum_{\underline{\ell},\underline{m}} 
(\underline{\ell}(\underline{\ell}+1))^{-\kappa}
\sum_{\substack{\ell,h,m,n}}
\frac{(-1)^m(-1)^n c^{N,\ell,\ell,m,-m}_{\underline{\ell},\underline{m}}
\overline{c^{N,h,h,n,-n}_{\underline{\ell},\underline{m}}}}
{\ell(\ell+1)h(h+1)}
\\
&\quad
+
\sum_{\underline{\ell},\underline{m}} 
(\underline{\ell}(\underline{\ell}+1))^{-\kappa}
\sum_{\substack{\ell,\ell',m,m'}}
\frac{c^{N,\ell,\ell',m,m'}_{\underline{\ell},\underline{m}}
\overline{c^{N,\ell,\ell',m,m'}_{\underline{\ell},\underline{m}}}}
{\ell^2(\ell+1)^2}
\\
&\quad
+
\sum_{\underline{\ell},\underline{m}} 
(\underline{\ell}(\underline{\ell}+1))^{-\kappa}
\sum_{\substack{\ell,\ell',m,m'}}
\frac{c^{N,\ell,\ell',m,m'}_{\underline{\ell},\underline{m}}
\overline{c^{N,\ell',\ell,m',m}_{\underline{\ell},\underline{m}}}}
{\ell(\ell+1)\ell'(\ell'+1)}
\\
&=
\sum_{\underline{\ell},0} 
(\underline{\ell}(\underline{\ell}+1))^{-\kappa}
\left|\sum_{\substack{\ell,m}}
\frac{(-1)^m c^{N,\ell,\ell,m,-m}_{\underline{\ell},0}
}
{\ell(\ell+1)} \right|^2
\\
&\quad
+
\sum_{\underline{\ell},\underline{m}} 
(\underline{\ell}(\underline{\ell}+1))^{-\kappa}
\sum_{\substack{\ell,\ell',m,\underline{m}-m}}
\frac{\left|c^{N,\ell,\ell',m,\underline{m}-m}_{\underline{\ell},\underline{m}}\right|^2}
{\ell(\ell+1)}
\left(
\frac{1}{\ell(\ell+1)} - \frac{1}{\ell'(\ell'+1)}
\right).
\end{align*}
Where we have used the fact that $\underline{m}=m+m'$.
We have the following equality of the $3j-$symbols (cfr. \autoref{sec:AppA})\footnote{This can be directly derived from the relation of the $3j-$symbols with the Clebsch-Gordan coefficients and the definition of the latter.}:
\begin{equation}
\sum_m(-1)^m\left( \begin{array}{ccc}
\ell & \ell & \underline{\ell} \\
m & -m & 0 
\end{array} \right)=\sqrt{2l+1}\delta^{\underline{\ell}}_0.
\end{equation}
Hence, the first term on the right hand side vanishes.
Therefore, we have:
\begin{equation}
\begin{array}{ll}
&\mathbb{E}^{\mu_N}
\left[
\| r^N(0) \|^2_{H^{-\kappa}(\su(N))}\right]
=
\sum_{\underline{\ell},\underline{m}} 
(\underline{\ell}(\underline{\ell}+1))^{-\kappa}
\sum_{\substack{\ell,\ell',m,\underline{m}-m}}
\frac{\left|c^{N,\ell,\ell',m,\underline{m}-m}_{\underline{\ell},\underline{m}}\right|^2}
{\ell(\ell+1)}
\left(
\frac{1}{\ell(\ell+1)} - \frac{1}{\ell'(\ell'+1)}
\right)\\
&=
\sum_{\underline{\ell},\underline{m}} 
(\underline{\ell}(\underline{\ell}+1))^{-\kappa}
\sum_{\substack{\ell,\ell',m,\underline{m}-m}}
\frac{\left|c^{N,\ell,\ell',m,\underline{m}-m}_{\underline{\ell},\underline{m}}\right|^2}
{2}
\left(
\frac{1}{\ell(\ell+1)} - \frac{1}{\ell'(\ell'+1)}
\right)^2\\
&=\sum_{\underline{\ell}=1}^{N-1}\sum_{\underline{m}=-\underline{\ell}}^{\underline{\ell}}
(\underline{\ell}(\underline{\ell}+1))^{-\kappa}
\sum_{\ell=1}^{N-1}\sum_{m=-\ell}^{\ell}\sum_{\ell'=|\underline{\ell}-\ell|+1}^{\min\lbrace N,\underline{\ell}+\ell\rbrace}
\frac{\left|c^{N,\ell,\ell',m,\underline{m}-m}_{\underline{\ell},\underline{m}}\right|^2}
{2}
\left(
\frac{1}{\ell(\ell+1)} - \frac{1}{\ell'(\ell'+1)}
\right)^2.
\end{array}
\end{equation}
We split the sum in two parts.
We say that $\ell \gg \underline{\ell}$ if $\ell \geq 2\underline{\ell} (\log(\underline{\ell})+1)$, and $\ell \approx \underline{\ell}$ if $\ell \leq 2\underline{\ell}(\log(\underline{\ell})+1)$.
Then, for $\ell \gg \underline{\ell}$ it holds $\ell' \geq |\underline{\ell}-\ell|+1 = \ell - \underline{\ell} + 1 \geq \ell/3$ and:
\begin{align*}
\left|\frac{1}{\ell(\ell+1)} - \frac{1}{\ell'(\ell'+1)}\right|^2
&=
\left|\frac{\ell'(\ell'+1)-\ell(\ell+1)}{\ell(\ell+1)\ell'(\ell'+1)}\right|^2
\\
&\leq C
\frac{\left|\ell'(\ell'+1)-\ell(\ell+1)\right|^2}{\ell^4(\ell+1)^4},
\end{align*}
where $C$ denotes, from now, on a positive suitable constant.
The numerator also satisfies:
\begin{align*}
\left|\ell'(\ell'+1)-\ell(\ell+1)\right|^2
&\leq
\max \left\{
\left|(\ell+\underline{\ell})(\ell+\underline{\ell}+1)-\ell(\ell+1)\right|^2 ,
\left|\ell(\ell+1) - (\ell-\underline{\ell}+1)(\ell-\underline{\ell}+2)\right|^2
\right\}.
\end{align*}
\begin{align*}
\left|(\ell+\underline{\ell})(\ell+\underline{\ell}+1)-\ell(\ell+1)\right|^2
&\leq
\left|
2\ell\underline{\ell}+\underline{\ell}^2+\underline{\ell}\right|^2
\leq
C \ell^2 \underline{\ell}^2 + C \underline{\ell}^4.
\end{align*}
\begin{align*}
\left|(\ell-\underline{\ell}+1)(\ell-\underline{\ell}+2)-\ell(\ell+1)\right|^2
&\leq
...
\leq
C \ell^2 \underline{\ell}^2 + C \underline{\ell}^4.
\end{align*}
Hence, for $\ell\gg\underline{\ell}$:
\begin{align*}
\left|\frac{1}{\ell(\ell+1)} - \frac{1}{\ell'(\ell'+1)}\right|^2
\leq C
\frac{\ell^2 \underline{\ell}^2 + \underline{\ell}^4}{\ell^4(\ell+1)^4}.
\end{align*}
By the Proposition~\ref{prop:str_const_diff} in \autoref{sec:AppA}, we get:
\begin{align*}
&\sum_{\underline{\ell}=1}^{N-1}\sum_{\underline{m}=-\underline{\ell}}^{\underline{\ell}}
(\underline{\ell}(\underline{\ell}+1))^{-\kappa}\sum_{\ell\gg\underline{\ell}}\sum_{m=-\ell}^{\ell}\sum_{\ell'=|\underline{\ell}-\ell|+1}^{\min\lbrace N,\underline{\ell}+\ell\rbrace}
\frac{\left|c^{N,\ell,\ell',m,\underline{m}-m}_{\underline{\ell},\underline{m}}\right|^2}
{2}
\frac{\ell^2 \underline{\ell}^2 + \underline{\ell}^4}{\ell^4(\ell+1)^4}
\\
&\leq 
\frac{C}{N^4}\sum_{\underline{\ell}=1}^{N-1}\sum_{\underline{m}=-\underline{\ell}}^{\underline{\ell}}
(\underline{\ell}(\underline{\ell}+1))^{-\kappa}\sum_{\ell\gg\underline{\ell}}\sum_{m=-\ell}^{\ell}\sum_{\ell'=|\underline{\ell}-\ell|+1}^{\min\lbrace N,\underline{\ell}+\ell\rbrace}
\frac{\ell^8 \underline{\ell}^4 + \ell^6\underline{\ell}^6}{\ell^4(\ell+1)^4}\\
&\leq 
\frac{C}{N^4}\sum_{\underline{\ell}=1}^{N-1}\sum_{\underline{m}=-\underline{\ell}}^{\underline{\ell}}
(\underline{\ell}(\underline{\ell}+1))^{-\kappa}\sum_{\ell\gg\underline{\ell}}\sum_{m=-\ell}^{\ell}
\underline{\ell}^5+\frac{\underline{\ell}^7}{\ell^2}\\
&\leq 
C\sum_{\underline{\ell}=1}^{N-1}\sum_{\underline{m}=-\underline{\ell}}^{\underline{\ell}}
(\underline{\ell}(\underline{\ell}+1))^{-\kappa}\left(\frac{\underline{\ell}^5}{N^2} + \frac{\underline{\ell}^7\log(N)}{N^4}\right)\\
&\leq 
C\left(\frac{N^{7-2\kappa}}{N^2} + \frac{N^{9-2\kappa}\log(N)}{N^4}\right)\\
&=CN^{5-2\kappa}\log(N),
\end{align*}
which goes to $0$ for $N\rightarrow\infty$ for $\kappa>5/2$.

For $\ell\approx\underline{\ell}$, $\ell'$ can be as small as $1$. 
Hence:
\begin{align*}
\left|\frac{1}{\ell(\ell+1)} - \frac{1}{\ell'(\ell'+1)}\right|^2
\leq C
\frac{\ell^2 \underline{\ell}^2 + \underline{\ell}^4}{\ell^2(\ell+1)^2}
\leq C
\frac{\underline{\ell}^4}{\underline{\ell}^2(\underline{\ell}+1)^2}
\leq C.
\end{align*}
By the Proposition~\ref{prop:str_const_diff} in \autoref{sec:AppA}, we get:
\begin{align*}
&\sum_{\underline{\ell}=1}^{N-1}\sum_{\underline{m}=-\underline{\ell}}^{\underline{\ell}}
(\underline{\ell}(\underline{\ell}+1))^{-\kappa}\sum_{\ell\approx\underline{\ell}}\sum_{m=-\ell}^{\ell}\sum_{\ell'=|\underline{\ell}-\ell|+1}^{\min\lbrace N,\underline{\ell}+\ell\rbrace}
\frac{\left|c^{N,\ell,\ell',m,\underline{m}-m}_{\underline{\ell},\underline{m}}\right|^2}
{2}
\left(\frac{1}{\ell(\ell+1)} - \frac{1}{\ell'(\ell'+1)}\right)^2
\\
&\leq 
\frac{C}{N^4}\sum_{\underline{\ell}=1}^{N-1}\sum_{\underline{m}=-\underline{\ell}}^{\underline{\ell}}
(\underline{\ell}(\underline{\ell}+1))^{-\kappa}\sum_{\ell\approx\underline{\ell}}\sum_{m=-\ell}^{\ell}\sum_{\ell'=|\underline{\ell}-\ell|+1}^{\min\lbrace N,\underline{\ell}+\ell\rbrace}
\underline{\ell}^6\ell^2\\
&\leq 
\frac{C}{N^4}\sum_{\underline{\ell}=1}^{N-1}\sum_{\underline{m}=-\underline{\ell}}^{\underline{\ell}}
(\underline{\ell}(\underline{\ell}+1))^{-\kappa}\sum_{\ell\approx\underline{\ell}}\sum_{m=-\ell}^{\ell}
\underline{\ell}^6\ell^3\\
&\leq 
\frac{C}{N^4}\sum_{\underline{\ell}=1}^{N-1}\sum_{\underline{m}=-\underline{\ell}}^{\underline{\ell}}
(\underline{\ell}(\underline{\ell}+1))^{-\kappa}\underline{\ell}^{11}\log^5(\underline{\ell})\\
&\leq 
C\frac{N^{11-2\kappa}\log^5(N)}{N^4}\\
&=CN^{7-2\kappa}\log^5(N),
\end{align*}
which goes to $0$ for $N\rightarrow\infty$ for $\kappa>7/2$.
\end{proof}

\section{Identification of the limit} \label{sec:limit}

\begin{proposition} \label{prop:conv}
Fix $\epsilon>0$. 
There exist a subsequence $(N_m)_{m \in \mathbb{N}}$, a common probability space $(\tilde{\Omega},\tilde{F},\tilde{\PP})$ and random variables $\tilde{\omega}^m, \tilde{\omega} : \tilde{\Omega} \to C([0,T],H^{-1-\epsilon}(\mathbb{S}^2))$, $m \in \N$ such that $\tilde{\omega}^m \sim \omega^{N_m}_{\omega_0}$ for every $m \in \N$ and $\tilde{\omega}^m \to \tilde{\omega}$ almost surely with respect to $\tilde{\PP}$.
\end{proposition}
\begin{proof}
Convergence in law up to a subsequence follows from \autoref{cor:bounds}, exploiting Simon compactness criterion \cite[Corollary 9]{Si1986} and Prokhorov Theorem.
Almost sure convergence in an auxiliary probability space is then a consequence of Skorokhod Theorem.
\end{proof}

%With a little abuse of notation we drop the tildes hereafter. 
In the following we say that a random variable taking values in $H^{-1-\epsilon}(\mathbb{S}^2)$ is a \emph{white noise} if distributed as $\nu$.
We recall the following result from \cite{Fl2018}, here adapted in order to consider functions defined on the sphere $\mathbb{S}^2$.
\begin{proposition}\cite[Theorem 8]{Fl2018}. \label{prop:Delort-Schochet}
Let $\omega: \Omega \to H^{-1-\epsilon}(\mathbb{S}^2)$ be a white noise, and for a fixed test function $\phi \in C^\infty(\mathbb{S}^2)$ denote
\begin{align*}
H_\phi(x,y) \coloneqq \frac12 K(x-y) (\nabla \phi(x) - \nabla \phi(y)).
\end{align*}
Assume to have a sequence of symmetric functions $H^N_\phi \in H^{2+2\epsilon}(\mathbb{S}^2 \times \mathbb{S}^2)$, $N \in \N$ that approximates $H_\phi$ in the following sense:
\begin{gather}
\lim_{N \to \infty} \label{eq:condition1}
\int_{\mathbb{S}^2} \int_{\mathbb{S}^2}
|H^N_\phi - H_\phi|^2 (x,y) d\mbox{vol}_S dy = 0;
\\
\lim_{N \to \infty} \label{eq:condition2}
\int_{\mathbb{S}^2}
H^N_\phi(x,x) d\mbox{vol}_S = 0.
\end{gather}
Then the sequence of random variables $\langle \langle \omega \otimes \omega , H^N_\phi \rangle \rangle$, $N \in \N$ is a Cauchy sequence in $L^2(\Omega)$. Moreover, the limit is independent of the sequence $H^N_\phi$, that is: if $\tilde{H}^N_\phi$, $N \in \N$ is another sequence satisfying \eqref{eq:condition1} and \eqref{eq:condition2}, then 
\begin{align*}
L^2(\Omega) - \lim_{N \to \infty}
\langle \langle \omega \otimes \omega , H^N_\phi \rangle \rangle
=
L^2(\Omega) - \lim_{N \to \infty}
\langle \langle \omega \otimes \omega , \tilde{H}^N_\phi \rangle \rangle.
\end{align*}
\end{proposition}
%\begin{proof}
%...
%\end{proof}

\begin{remark} \label{rmk:existence}
There exists a sequence $H^N_\phi$ satisfying \eqref{eq:condition1} and \eqref{eq:condition2}.
It can be constructed by mollification of the Biot-Savart kernel:
\begin{align*}
H^N_\phi(x,y) \coloneqq \frac12 K_{1/N}(x-y) (\nabla \phi(x) - \nabla \phi(y)),
\end{align*}
see \cite[Remark 9]{Fl2018} for details.
\end{remark}

\begin{definition} \label{def:diamond}
Let $\omega: \Omega \to H^{-1-\epsilon}(\mathbb{S}^2)$ be a white noise and take $\phi \in C^\infty(\mathbb{S}^2)$.
We define the random variable $\langle \omega \diamond \omega , H_\phi \rangle \in L^2(\Omega)$ as the $L^2(\Omega)$-limit of any sequence $\langle \langle \omega \otimes \omega , H^N_\phi \rangle \rangle$, $N \in \N$, with $H^N_\phi$ satisfying properties \eqref{eq:condition1} and \eqref{eq:condition2}. 
\end{definition}

We have all the necessary tools to state and prove our main result, characterizing the law of any accumulation point of the sequence $(\omega_{\omega_0}^{N_m})_{m \in \N}$.
In view of \autoref{prop:Delort-Schochet}, we can interpret \eqref{eq:euler_symm} below as a symmetrized version of Euler equations.

\begin{theorem}
Fix $\epsilon>0$, and let $\tilde{\omega}: \tilde{\Omega} \to C([0,T],H^{-1-\epsilon}(\mathbb{S}^2))$ be given by \autoref{prop:conv}. 
Then for every test function $\phi \in C^\infty(\mathbb{S}^2)$ it holds $\tilde{\mathbb{P}}$-a.s. 
\begin{align} \label{eq:euler_symm}
\langle \tilde{\omega}_t , \phi \rangle
=
\langle \tilde{\omega}_0 , \phi \rangle
+
\int_0^t 
\langle \tilde{\omega}_s \diamond \tilde{\omega}_s , H_\phi \rangle ds,
\qquad
\forall t \in [0,T].
\end{align}
\end{theorem}

\begin{proof}
Let $\tilde{\omega}^m, \tilde{\omega}$ be given by \autoref{prop:conv}, and fix $\phi \in C^\infty(\mathbb{S}^2)$.
Recalling \eqref{eq:omega_dyn}, it is easy to check for every $m \in \N$ and $\tilde{\PP}$-a.s.
\begin{align*}
\langle \tilde{\omega}^m_t , \phi \rangle
-
\langle \tilde{\omega}^m_0 , \phi \rangle
&=
\int_0^t 
\langle \langle \tilde{\omega}^m_s \otimes \tilde{\omega}^m_s , H_\phi \rangle \rangle ds 
\\
&\quad+
\int_0^t 
\langle \langle \tilde{\omega}^m_s \otimes \tilde{\omega}^m_s , H_{j_{N_m} \Pi_{N_m} \phi} - H_\phi \rangle \rangle ds 
+
\int_0^t
\langle \tilde{r}^{N_m}_s , \phi \rangle ds
\end{align*}
for every $t \in [0,T]$, where $\tilde{r}^{N_m}$ is distributed as $j_{N_m} r^{N_m}$.

Since $\tilde{\omega}^m \to \tilde{\omega}$ as $m \to \infty$ almost surely with respect to the $C([0,T],H^{-1-\epsilon})$ topology, we have
\begin{align*}
\langle \tilde{\omega}^m_t , \phi \rangle
-
\langle \tilde{\omega}^m_0 , \phi \rangle
\to
\langle \tilde{\omega}_t , \phi \rangle
-
\langle \tilde{\omega}_0 , \phi \rangle
\quad\mbox{as } m \to \infty,
\end{align*}
with probability one. Concerning the second summand on the right-hand-side, we notice that $H_{j_{N_m} \Pi_{N_m} \phi} - H_\phi$ converges to zero in $L^2(\mathbb{S}^2 \times \mathbb{S}^2)$ and therefore
\begin{gather*}
\tildeE{\left| \int_0^t 
\langle \langle \tilde{\omega}^m_s \otimes \tilde{\omega}^m_s , H_{j_{N_m} \Pi_{N_m} \phi} - H_\phi \rangle \rangle ds \right|}
\\
\leq C \tildeE{\|\tilde{\omega}^m_s \otimes \tilde{\omega}^m_s\|_{L^2(\mathbb{S}^2\times \mathbb{S}^2)}} \|H_{j_{N_m} \Pi_{N_m} \phi} - H_\phi\|_{L^2(\mathbb{S}^2\times \mathbb{S}^2)}
\\
=C \mathbb{E}^{\nu_{N_m}} \left[ \|\omega_0 \otimes \omega_0\|_{L^2(\mathbb{S}^2\times \mathbb{S}^2)} \right] \|H_{j_{N_m} \Pi_{N_m} \phi} - H_\phi\|_{L^2(\mathbb{S}^2\times \mathbb{S}^2)} \to 0
\end{gather*}
as $m \to \infty$, which implies the almost sure convergence up to a subsequence (that we still denote $m$ with a little abuse of notation):
\begin{align*}
\int_0^t 
\langle \langle \tilde{\omega}^m_s \otimes \tilde{\omega}^m_s , H_{j_{N_m} \Pi_{N_m} \phi} - H_\phi \rangle \rangle ds \to 0.
\end{align*}
Similarly,
\begin{align*}
\tildeE{\left| \int_0^t 
\langle \tilde{r}^{N_m}_s , \phi \rangle ds \right|}
&\leq
C \tildeE{ \| \tilde{r}^{N_m}_s \|_{H^{-\kappa}(\mathbb{S}^2)}} \| \phi \|_{H^{\kappa}(\mathbb{S}^2)} 
\\
&=
C \mathbb{E}^{\mu_{N_m}}{ \| r^{N_m}_s \|_{H^{-\kappa}(\su(N))}} \| \phi \|_{H^{\kappa}(\mathbb{S}^2)} \to 0
\end{align*}
as $m \to \infty$, which implies the almost sure convergence up to a subsequence (that we still denote $m$ with a little abuse of notation):
\begin{align*}
\int_0^t 
\langle \tilde{r}^{N_m}_s , \phi \rangle ds \to 0.
\end{align*}

Finally, let us focus on the first term on the right-hand-side.
Let $H^M_\phi$, $M \in \N$, be a sequence of $H^{2+2\epsilon}(\mathbb{S}^2 \times \mathbb{S}^2)$ functions that approximates $H_\phi$ in the sense of \autoref{prop:Delort-Schochet} above, and exists by \autoref{rmk:existence}.
We can decompose, for fixed $M \in \N$:
\begin{align*}
\int_0^t 
\langle \langle \tilde{\omega}^m_s \otimes \tilde{\omega}^m_s , H_\phi \rangle \rangle ds 
&=
\int_0^t 
\langle \langle \tilde{\omega}^m_s \otimes \tilde{\omega}^m_s , H_\phi - H^M_\phi \rangle \rangle ds 
\\
&\quad+
\int_0^t 
\langle \langle \tilde{\omega}^m_s \otimes \tilde{\omega}^m_s - \tilde{\omega}_s \otimes \tilde{\omega}_s , H^M_\phi \rangle \rangle ds 
\\
&\quad+
\int_0^t 
\langle \langle \tilde{\omega}_s \otimes \tilde{\omega}_s , H^M_\phi \rangle \rangle ds. 
\end{align*}
Now, by condition \eqref{eq:condition1} for every $\delta>0$ there exists $M \in \N$ such that
\begin{align*}
\tildeE{\left|\int_0^t 
\langle \langle \tilde{\omega}^m_s \otimes \tilde{\omega}^m_s , H_\phi-H^M_\phi \rangle \rangle ds \right|} 
\\
\leq 
C
\tildeE{\| \tilde{\omega}^m_s \otimes \tilde{\omega}^m_s \|_{L^2(\mathbb{S}^2 \times \mathbb{S}^2)}}
\| H_\phi-H^M_\phi \|_{L^2(\mathbb{S}^2 \times \mathbb{S}^2)}
\leq\delta;
\end{align*}
moreover, since it is easy to check that $\tilde{\omega}$ is a white noise by \autoref{cor:WN}, by \autoref{prop:Delort-Schochet} and \autoref{def:diamond} for every $\delta>0$ there exists $M \in \N$ such that
\begin{align*}
\tildeE{\left|\int_0^t 
\langle \langle \tilde{\omega}_s \otimes \tilde{\omega}_s , H^M_\phi \rangle \rangle ds-\int_0^t 
\langle \tilde{\omega}_s \diamond \tilde{\omega}_s , H_\phi \rangle ds\right|} \leq \delta.
\end{align*}
Having fixed such $M$, we have
\begin{align*}
&\tildeE{\left|\int_0^t 
\langle \langle \tilde{\omega}^m_s \otimes \tilde{\omega}^m_s , H_\phi \rangle \rangle ds-\int_0^t 
\langle \tilde{\omega}_s \diamond \tilde{\omega}_s , H_\phi \rangle ds\right|} 
\leq 
\\
&\quad
\tildeE{\left|\int_0^t 
\langle \langle \tilde{\omega}^m_s \otimes \tilde{\omega}^m_s-\tilde{\omega}_s \otimes \tilde{\omega}_s , H^M_\phi \rangle \rangle ds\right|} + 2\delta \to 2\delta
\end{align*}
as $m \to \infty$, since $H^M_\phi \in H^{2+2\epsilon}(\mathbb{S}^2 \times \mathbb{S}^2)$ and $\tilde{\omega}^m \to \tilde{\omega}$ in $C([0,T],H^{-1-\epsilon}(\mathbb{S}^2))$, implying $\tilde{\omega}^m \otimes \tilde{\omega}^m \to \tilde{\omega} \otimes \tilde{\omega}$ in $C([0,T],H^{-2-2\epsilon}(\mathbb{S}^2 \otimes \mathbb{S}^2))$.
Since $\delta$ is arbitrary, we deduce
\begin{align*}
&\tildeE{\left|\int_0^t 
\langle \langle \tilde{\omega}^m_s \otimes \tilde{\omega}^m_s , H_\phi \rangle \rangle ds-\int_0^t 
\langle \tilde{\omega}_s \diamond \tilde{\omega}_s , H_\phi \rangle ds\right|} 
\to 0
\end{align*}
as $m \to \infty$, that yields the almost sure convergence 
\begin{align*}
\int_0^t 
\langle \langle \tilde{\omega}^m_s \otimes \tilde{\omega}^m_s , H_\phi \rangle \rangle ds \to \int_0^t 
\langle \tilde{\omega}_s \diamond \tilde{\omega}_s , H_\phi \rangle ds
\end{align*}
up to subsequences. Putting all together we have shown
\begin{align*}
\langle \tilde{\omega}_t , \phi \rangle
-
\langle \tilde{\omega}_0 , \phi \rangle
=
\int_0^t 
\langle \tilde{\omega}_s \diamond \tilde{\omega}_s , H_\phi \rangle ds,
\qquad
\forall t \in [0,T],
\end{align*}
and the proof is complete.
\end{proof}

\section{Open problems} \label{sec:open}

\subsection{Gibbs measure associated to Casimirs}
The 2D Euler equations on a compact surface $S$ have inifinitely many conservation laws.
The following integrals, when defined, are invariants for the dynamics:
\begin{align*} 
H\left(  \omega\right)    & =\int_S\psi\omega d\mbox{vol}_S\\
C_{f}\left(  \omega\right)    & =\int_S f(\omega)d\mbox{vol}_S,
\end{align*}
where $f:\Rr\rightarrow\Rr$ can be any $C^1$ function.
In particular, for $f(x)=x^2$, we have the enstrophy $E\left(  \omega\right) =\int \omega^2d\mbox{vol}_S$.
The presence of these conservation laws comes from the fact that 2D Euler equations are an infinite dimensional Lie--Poisson system on the dual of the Lie algebra of smooth divergence-free vector fields on $S$ \cite{ArKh1998}.
This space can be identified with the space of smooth functions on $S$.
Therefore, because of the Hamiltonian nature of the Euler equations, we formally have that the "flat measure" on $C^{\infty}(S)$ is an invariant measure.
Hence, heuristically we can define the following family of invariant measures for $\alpha,\beta,\gamma_p\geq 0$:
\[
\mu\left(  d\omega\right)  =Z^{-1}\exp\left(  -\alpha E\left(  \omega\right)
-\beta H\left(  \omega\right)  -\sum_{p>2}\gamma_p C_{p}\left(
\omega\right)  \right)  \left[  d\omega\right]
\]
where
\begin{align*}
C_{p}\left(  \omega\right)    & =\int_S\omega^{p}d\mbox{vol}_S,
\end{align*}
$\left[  d\omega\right]  $ is the formal "flat measure" on $C^{\infty}(S)$ and $Z$ is the partition function.
In order to make this more rigorous, we cannot use the formal "flat measure" $\left[  d\omega\right]  $.
Instead, we take the enstrophy measure $\nu$ as reference measure on $H^{-1-}(S)=\cap_{\epsilon>o}H^{-1-\epsilon}(S)$ (cfr. Section~\autoref{sec:gaussian_measures}).  
We could then define $\mu$ as:
\[
\mu\left(  d\omega\right)  =\widetilde{Z}^{-1}\exp\left(  -\beta H\left(
\omega\right)  -\sum_{i}\gamma_p C_{p}\left(  \omega\right)  \right)
\nu\left(  d\omega\right)
\]
where
\[
\widetilde{Z}:=\int\exp\left(  -\beta H\left(  \omega\right)  -\sum_{i}%
\gamma_{p}C_{p}\left(  \omega\right)  \right)  \nu\left(
d\omega\right)  .
\]
We notice that the measure $\mu$ for $\gamma_p=0$ can be defined using the theory of Gaussian measures on $H^{-1-}(S)$. 
However, for instance, taking $\beta=0$ and $\gamma_p\neq 0$ only for $p=4$, the measure
\begin{align*}
\mu\left(  d\omega\right)    & =\widetilde{Z}^{-1}\exp\left(  -\gamma
\int_S\omega^{4}d\mbox{vol}_S\right)  \nu\left(  d\omega\right)  \\
\widetilde{Z}  & :=\int\exp\left(  -\gamma\int_S\omega^{4}d\mbox{vol}_S\right)  \nu\left(  d\omega\right)
\end{align*}
it is not well defined, since we do not have a precise meaning of a power of an element in $H^{-1-}(S)$.
In order to make sense of this operation, one would like to use the \textit{renormalization} theory, that allows to define the renormalized power  
\[
:\int_S\omega^{4}d\mbox{vol}_S:
\]
of a suitable Gaussian measure $\omega$ as the mean square limit of the renormalized power 
\[
\int\omega_{\varepsilon}^{4}d\mbox{vol}_S - 6C_\varepsilon \int\omega_{\varepsilon}^{2}d\mbox{vol}_S + 3C_\varepsilon^2
\]
where $\omega_{\varepsilon}$ is a mollification of $\omega$ and $C_\varepsilon \to \infty$ is a suitable renormalization constant.
Unfortunately, the current renormalization theory does not cover Gaussian measures associated with Casimirs higher than the enstrophy.

The quantized Euler equations \eqref{eq:euler} in $\su(N)$ have the following invariants:
\begin{align*}
H\left( W\right)    & =\Tr(PW)\\
C_{p}\left(  W\right)    & =\Tr(W^p),
\end{align*}
for $p=2,\ldots,N$.
It is known that for smooth $\omega$, we get \cite{MiWeCr1992}:
\begin{align*}
H\left(\Pi_N\omega\right)    &\rightarrow H(\omega)\\
C_{p}\left(  \Pi_N\omega\right)    & \rightarrow C_{p}\left( \omega\right) ,
\end{align*}
for $N\rightarrow\infty$. 
In section~\ref{sec:gaussian_measures}, we have seen that $\nu_N \rightharpoonup \nu$ as measures on $H^{-1-}(\mathbb{S}^2)$.
Let, for instance $p=4$.
Defining
\begin{align*}
\eta_N\left(  dW\right)    & =\widetilde{Z}_N^{-1}\exp\left(  -\gamma
C_{p}\left(W\right)\right)  \mu_N\left(  dW\right)  \\
\widetilde{Z}_N  & :=\int\exp\left(  -\gamma C_{p}\left(W\right)\right)  \mu_N\left(  dW\right),
\end{align*}
we would like to show that $j_N^*\eta_N$ has a weak limit in $H^{-1-}(\mathbb{S}^2)$.

\subsection{Line integrals and Kelvin theorem}
Developing the machinery needed to prove invariance theorems based on line integrals is also an appealing question, having in mind especially Kelvin theorem, see \cite{MaPu1994}. 
In the generalized setting of the enstrophy measure, where all fields are distributional, this looks a formidable task, still open. 
However, we would like to emphasize that line integrals on deterministic curves are well defined, in spite of an apparent difficulty. 
It is the generalization to random curves which is open and, unfortunately, necessary to develop invariance properties, since one should consider curves moving with the fluid, hence random.

Let us thus show that line integrals are well defined on deterministic closed curves. 
We follow the approach developed for the definition of line integrals of the Gaussian Free Field, see for instance \cite{HuMiPe2010}.
Let us restrict ourselves for simplicity to curves which are
boundaries of bounded open connected sets $A\subset\mathbb{S}^{2}$. 
Assume
that $\partial A$ is a Lipschitz boundary and assume that $\gamma:\left[
a,b\right]  \rightarrow\mathbb{S}^{2}$ is a Lipschitz continuous curve
parametrizing $\partial A$. 
Assume that the parametrization is regular, namely
that the derivative $\gamma^{\prime}\left(  t\right)  $, which exists a.s.,
has the property $\left\vert \gamma^{\prime}\left(  t\right)  \right\vert \geq
c>0$ a.s., for some positive constant $c$. It is known that the map%
\[
f\mapsto f|_{\partial A}%
\]
originally defined on $W^{s,2}\left(  \mathbb{S}^{2}\right)  \cap C\left(
\mathbb{S}^{2}\right)  $, for some $s>\frac{1}{2}$, extends to a bounded
linear map from $W^{s,2}\left(  \mathbb{S}^{2}\right)  $ to $L^{2}\left(
\partial A\right)  $ (in fact it takes values in $W^{s-\frac{1}{2},2}\left(
\partial A\right))$. Thanks to regularity of $\gamma$, we can say that the
function
\begin{equation}
t\mapsto f\left(  \gamma\left(  t\right)  \right)  \text{ is of class }%
L^{2}\left(  a,b\right)  \text{, for every }f\in%
%TCIMACRO{\dbigcap \limits_{s>\frac{1}{2}}}%
%BeginExpansion
{\displaystyle\bigcap\limits_{s>\frac{1}{2}}}
%EndExpansion
W^{s,2}\left(  \mathbb{S}^{2}\right)  .\label{integrability along curves}%
\end{equation}
Moreover, for every $s>\frac{1}{2}$ there is a constant $C_{s}>0$ such that%
\begin{equation}
\left\Vert f\circ\gamma\right\Vert _{L^{2}\left(  a,b\right)  }\leq
C_{s}\left\Vert f\right\Vert _{W^{s,2}\left(  \mathbb{S}^{2}\right)
}.\label{continuity of trace}%
\end{equation}

Associated to the rectifiable curve $\gamma$ we may define, for every
$s>\frac{1}{2}$, the \textit{rectifiable current}
\[
\Gamma:W^{s,2}\left(  \mathbb{S}^{2},\mathbb{R}^{2}\right)  \rightarrow
\mathbb{R}%
\]
defined as%
\[
\Gamma\left(  v\right)  =\int_{a}^{b}v\left(  \gamma\left(  t\right)  \right)
\cdot\gamma^{\prime}\left(  t\right)  dt
\]
for every $v\in W^{s,2}\left(  \mathbb{S}^{2},\mathbb{R}^{2}\right)  $. Indeed
notice that, by (\ref{integrability along curves}) the integral is finite and
by (\ref{continuity of trace}) the map $\Gamma$ is bounded. Thus $\Gamma$ is a
bounded linear functional on $W^{s,2}\left(  \mathbb{S}^{2},\mathbb{R}%
^{2}\right)  $, namely it is an element of the dual of $W^{-s,2}\left(
\mathbb{S}^{2},\mathbb{R}^{2}\right)  $, and this holds for every $s>\frac
{1}{2}$:
\begin{equation}
\Gamma\in%
%TCIMACRO{\dbigcap \limits_{s>\frac{1}{2}}}%
%BeginExpansion
{\displaystyle\bigcap\limits_{s>\frac{1}{2}}}
%EndExpansion
W^{-s,2}\left(  \mathbb{S}^{2},\mathbb{R}^{2}\right)
.\label{regularity of the current}%
\end{equation}

Let now $\mu$ be the enstrophy measure on $\mathbb{S}^{2}$ defined in Section~\autoref{sec:gaussian_measures}, namely the centered Gaussian measure, supported on
$W^{-1-\epsilon,2}\left(  \mathbb{S}^{2},\mathbb{R}\right)  $ with identity
covariance%
\[
\int_{W^{-1-\epsilon,2}\left(  \mathbb{S}^{2},\mathbb{R}\right)  }\left\langle
\omega,\varphi\right\rangle \left\langle \omega,\psi\right\rangle \mu\left(
d\omega\right)  =\left\langle \varphi,\psi\right\rangle
\]
for all $\varphi,\psi\in W^{1+\epsilon,2}\left(  \mathbb{S}^{2},\mathbb{R}%
\right)  $, where $\left\langle \cdot,\cdot\right\rangle $ inside the integral
is the dual pairing, outside the scalar product in $L^{2}\left(
\mathbb{S}^{2},\mathbb{R}\right)  $. Let $K$ be the Biot-Savart map from
$W^{-1-\epsilon,2}\left(  \mathbb{S}^{2},\mathbb{R}\right)  $ to
$W^{-\epsilon,2}\left(  \mathbb{S}^{2},\mathbb{R}^{2}\right)  $ and let (we
use the notation $K\ast$ interpreting $K$ as a kernel)
\[
\xi=K\ast\mu
\]
be the centered Gaussian velocity field associated to the enstrophy measure,
namely a centered Gaussian measure, supported on $W^{-\epsilon,2}\left(
\mathbb{S}^{2},\mathbb{R}^{2}\right)  $, such that%
\begin{align*}
& \int_{W^{-\epsilon,2}\left(  \mathbb{S}^{2},\mathbb{R}\right)  }\left\langle
v,w\right\rangle \left\langle v,z\right\rangle \xi\left(  dv\right)  \\
& =\int_{W^{-1-\epsilon,2}\left(  \mathbb{S}^{2},\mathbb{R}\right)
}\left\langle K\ast\omega,w\right\rangle \left\langle K\ast\omega
,z\right\rangle \mu\left(  d\omega\right)  \\
& =\left\langle K^{\prime}\ast w,K^{\prime}\ast z\right\rangle =\left\langle
K\ast K^{\prime}\ast w,z\right\rangle
\end{align*}
for all $w,z\in W^{\epsilon,2}\left(  \mathbb{S}^{2},\mathbb{R}^{2}\right)  $,
where $K^{\prime}$ denotes the dual of $K$. One can recognize that $K\ast
K^{\prime}\ast$ is $\left(  -\Delta\right)  ^{-1}$, hence%
\[
\int_{W^{-\epsilon,2}\left(  \mathbb{S}^{2},\mathbb{R}\right)  }\left\langle
v,w\right\rangle \left\langle v,z\right\rangle \xi\left(  dv\right)
=\left\langle \left(  -\Delta\right)  ^{-1}w,z\right\rangle .
\]

Formally we aim to define
\[
\left\langle \xi,\Gamma\right\rangle =\int_{a}^{b}\xi\left(  \gamma\left(
t\right)  \right)  \cdot\gamma^{\prime}\left(  t\right)  dt.
\]
The key remark is that the covariance property above of the measure
$\xi\left(  dv\right)  $ allows to extend the definition of the Gaussian
random variable $\left\langle v,w\right\rangle $, $v$ selected by $\xi\left(
dv\right)  $, from vector fields $w$ of class $W^{\epsilon,2}\left(
\mathbb{S}^{2},\mathbb{R}^{2}\right)  $ to vector fields of class
$W^{-1,2}\left(  \mathbb{S}^{2},\mathbb{R}^{2}\right)  $, which includes the
space where $\Gamma$ lives, see (\ref{regularity of the current}). 

\begin{proposition}
Under the measure $\xi\left(  dv\right)  $, if $w\in W^{-1,2}\left(
\mathbb{S}^{2},\mathbb{R}^{2}\right)  $ a centered Gaussian random variable
$\left\langle v,w\right\rangle $ is well defined, with variance $\left\langle
\left(  -\Delta\right)  ^{-1}w,w\right\rangle $. Since the rectifiable current
$\Gamma$, associated to a regular Lipschitz curve $\gamma:\left[  a,b\right]
\rightarrow\mathbb{S}^{2}$ as done above, is of class
(\ref{regularity of the current}), the r.v. $\left\langle v,\Gamma
\right\rangle $ is well defined and we take it as the definition of $\int%
_{a}^{b}\xi\left(  \gamma\left(  t\right)  \right)  \cdot\gamma^{\prime
}\left(  t\right)  dt$.
\end{proposition}

Let us explain why the Gaussian random variable $\left\langle v,w\right\rangle $ is well
defined also for $w\in W^{-1,2}\left(  \mathbb{S}^{2},\mathbb{R}^{2}\right)
$. First, a fast but formal explanation: if $w,z\in W^{-1,2}\left(
\mathbb{S}^{2},\mathbb{R}^{2}\right)  $, then $\left(  -\Delta\right)
^{-1}w\in W^{1,2}\left(  \mathbb{S}^{2},\mathbb{R}^{2}\right)  $ and the dual
pairing $\left\langle \left(  -\Delta\right)  ^{-1}w,z\right\rangle $ is well
defined. 

More rigorously, if $\theta_{\epsilon}\left(  x\right)  =\epsilon^{-2}%
\theta\left(  \epsilon^{-1}x\right)  $ is a family if classical smooth
symmetric mollifiers on $\mathbb{S}^{2}$, if $w\in W^{-1,2}\left(
\mathbb{S}^{2},\mathbb{R}^{2}\right)  $, standing that $\theta_{\epsilon}\ast
w\in W^{\epsilon,2}\left(  \mathbb{S}^{2},\mathbb{R}^{2}\right)  $, we have%
\begin{align*}
& \int_{W^{-\epsilon,2}\left(  \mathbb{S}^{2},\mathbb{R}\right)  }\left(
\left\langle v,\theta_{\epsilon}\ast w\right\rangle-\left\langle v,\theta_{\epsilon^{\prime}}\ast w\right\rangle \right)
^{2}\xi\left(  dv\right)  \\
& =\left\langle \left(  -\Delta\right)  ^{-1}\left(  \theta_{\epsilon}%
-\theta_{\epsilon^{\prime}}\right)  \ast w,\left(  \theta_{\epsilon}%
-\theta_{\epsilon^{\prime}}\right)  \ast w\right\rangle \\
& =\left\langle \left(  \theta_{\epsilon}-\theta_{\epsilon^{\prime}}\right)
\ast\left(  \theta_{\epsilon}-\theta_{\epsilon^{\prime}}\right)  \ast\left(
-\Delta\right)  ^{-1}w,w\right\rangle
\end{align*}
which implies (by the convergence properties of $\theta_{\epsilon}\ast$ in
$W^{1,2}\left(  \mathbb{S}^{2},\mathbb{R}^{2}\right)  $) that the family
$\left\langle v,\theta_{\epsilon}\ast w\right\rangle $ is Cauchy in $L^{2}$
with respect to the measure $\xi\left(  dv\right)  $. We call $\left\langle
v,w\right\rangle $ its limit, which is a centered
Gaussian r.v. with variance equal to $\left\langle \left(  -\Delta\right)
^{-1}w,w\right\rangle $. 

These properties are based on the fact that $\Gamma$ is deterministic. As said
at the beginning, the extension to random curves is an open problem.

\vspace{.5cm}
Within the quantized Euler equations \eqref{eq:euler} in $\su(N)$, it is possible to identify the discrete analogue of the line integrals of the velocity field.
Alternatively to the usual choice for the Casimirs $C^N_n(W)=Tr(W^n)$, for $n=2,\ldots,N$, one can equivalently consider the eigenvalues $\lambda_i$ of $W$.
Indeed it holds
\[Tr(W^n)=\sum_{i=1}^N \lambda_i^n.\]
The first choice of the Casimirs corresponds to a discrete version of the momenta of the continuous vorticity $C_n(\omega)=\int_{\Ss^2} \omega^n dS$, for $n>1$, whereas the second one corresponds to the conserved quantities given by the Kelvin circulation theorem.
Indeed, we have that the Kelvin circulation theorem implies that for any material domain  $A(t)\in\Ss^2$, i.e. a domain $A=A(t)$ evolving accordingly to the fluid motion, the integral $\int_{A(t)} \omega dS=\int_{\partial A(t)}\textbf{u}\cdot d\textbf{s}$ is invariant in time, where $\nabla\times\textbf{u}=\omega\textbf{n}$, for $\textbf{n}$ normal vector on $\Ss^2$.

We now want to show the heuristic analogy among the eigenvalues of $W$ and the integrals $\int_{A(t)} \omega dS$.
Let us consider then spectral decomposition of $W$:
\[W=E\Lambda E^*,\]
for $E$ unitary and $\Lambda$ purely imaginary diagonal.
Let $e_i$, for $i=1,\ldots,N$ be the columns of $E$ and the $\lambda_i$ the eigenvalues of $W$.
Then we can write:
\[
W=\sum_{i=1}^N \lambda_i e_ie_i^*.
\]
The matrices $e_ie_i^*$ are pairwise orthogonal with respect to the Frobenius inner product.
Hence, 
\[Tr(W^*e_ie_i^*)=\lambda_i.\]
The heuristic analogy with the Kelvin's theorem reads as:
\[\mbox{i} Tr(W^*e_ie_i^*)\approx\int_{A(t)} \omega dS,\]
for some domain $A(t)$ which corresponds to the support of $j_N (\mbox{i}e_ie_i^*)\in C^\infty(\Ss^2)$, for $N\rightarrow\infty$.

Analogously for the other choice of Casimirs, we can define the invariant measure on $\su(N)$ as
\begin{align*}
\eta_N\left(  dW\right)    & =\widetilde{Z}_N^{-1}\exp\left(  -\gamma
\Tr(W^*\mbox{i}e_ie_i^*)^2\right)  \mu_N\left(  dW\right)  \\
\widetilde{Z}_N  & :=\int\exp\left(  -\gamma\Tr(W^*\mbox{i}e_ie_i^*)^2\right)  \mu_N\left(  dW\right),
\end{align*}
we would like to show that $j_N^*\eta_N$ has a weak limit in $H^{-1-}(\mathbb{S}^2)$.

\appendix

\section{Structure constants estimates for the 2-sphere} \label{sec:AppA}
Let $N$ be a positive integer. Let $C^{(N)\underline{\ell}\underline{m}}_{\ell m ,\ell' m'}$ and $C^{\underline{\ell}\underline{m}}_{\ell m ,\ell' m'}$ be respectively the structure constants of $\su(N)$ with respect to the $T^N_{\ell,m}$ basis and $C^\infty(\Ss^2)$ with respect to the $Y_{\ell,m}$ basis and the Poisson bracket \eqref{eq:poisson_brack_sph} (see \cite{RiSt2014}).
Then we have that\footnote{$\left(\begin{smallmatrix}\cdot&\cdot&\cdot\\\cdot&\cdot&\cdot\end{smallmatrix}\right)$
are the Wigner 3j-symbols and $\lbrace\begin{smallmatrix}\cdot&\cdot&\cdot\\\cdot&\cdot&\cdot\end{smallmatrix}\rbrace$
 are the  Wigner 6j-symbols.}:
\begin{align*}
C^{(N)\underline{\ell}\underline{m} }_{\ell m ,\ell' m'}=&(N+1)^{3/2}(1-(-1)^{\ell +\ell' +\underline{\ell}})(-1)^{N+\underline{m}}\sqrt{2\ell +1}\sqrt{2\ell' +1}\sqrt{2\underline{\ell}+1} \cdot \\
&\cdot\left( \begin{array}{ccc}
\ell  & \ell'  & \underline{\ell} \\
m  & m' & -\underline{m} 
\end{array} \right)
\Bigg\lbrace \begin{array}{ccc}
\ell  & \ell'  & \underline{\ell} \\
\frac{N}{2} & \frac{N}{2} & \frac{N}{2} 
\end{array} \Bigg\rbrace \\
=& (N+1)(1-(-1)^{\ell +\ell' +\underline{\ell}})\sqrt{2\ell +1}\sqrt{2\ell' +1}\sqrt{2\underline{\ell}+1} \cdot \\
&\cdot\left( \begin{array}{ccc}
\ell  & \ell'  & \underline{\ell} \\
m  & m' & -\underline{m} 
\end{array} \right) \ell !\ell' !\underline{\ell}! \Delta(\ell ,\ell' ,\underline{\ell}) \prod_{p_1=1}^{\ell}\prod_{p_2=1}^{\ell'}\prod_{p_3=1}^{\underline{\ell}}\Big(1-\Big(\frac{p_i}{N+1}\Big)^2\Big)^{-1/2}\cdot \\
&\cdot \sum_{k=\max\lbrace\ell,\ell',\underline{\ell}\rbrace}^{min\lbrace\ell+\ell',\ell'+\underline{\ell},\ell+\underline{\ell}\rbrace}\dfrac{(-1)^k S(k,L,N)}{R(\ell ,\ell' ,\underline{\ell},k)},
\end{align*}
where
\begin{align*}
\Delta(\ell ,\ell' ,\underline{\ell})=&\sqrt{\dfrac{(\ell +\ell' -\underline{\ell})!(\ell -\ell' +\underline{\ell})!(-\ell +\ell' +\underline{\ell})!}{(l+\ell' +\underline{\ell}+1)!}}\\
S(k,L,N)=&\prod_{i=k}^{k-L} \Big(1+\frac{i}{N+1}\Big) \\
R(\ell ,\ell' ,\underline{\ell},k)=&(k-\ell)!(k-\ell')!(k-\underline{\ell})!(\ell+\ell'-k)!(\ell'+\underline{\ell}-k)!(\ell+\underline{\ell}-k)!\\
L=&\ell +\ell' +\underline{\ell},
\end{align*}
and
\begin{align*}
C_{\ell m  \ell' m'}^{\underline{\ell} \underline{m}}=&(1-(-1)^{\ell +\ell' +\underline{\ell}})(-1)^{\underline{m}+1}\sqrt{2l+1}\sqrt{2\ell' +1}\sqrt{2\underline{\ell}+1}\cdot \\ 
&\cdot 
\left( \begin{array}{ccc}
\ell  & \ell'  & \underline{\ell} \\
m  & m' & -\underline{m}  
\end{array} \right)P(\ell ,\ell' ,\underline{\ell}),
\end{align*}
where, for odd values of $L=\ell +\ell' +\underline{\ell}$,
\begin{align*}
P(\ell ,\ell' ,\underline{\ell})=&(-1)^{(\ell +\ell' -\underline{\ell}+1)/2}\Delta(\ell ,\ell' ,\underline{\ell})(\ell +\ell' +\underline{\ell}+1)\cdot\\
&\cdot\dfrac{((\ell +\ell' +\underline{\ell}-1)/2)!}{((-\ell +\ell' +\underline{\ell}-1)/2)!((\ell -\ell' +\underline{\ell}-1)/2)!((\ell +\ell' -\underline{\ell}-1)/2)!}.
\end{align*}
Note that for even values of $L=\ell +\ell' +\underline{\ell}$, P may be arbitrarily defined.

Developing $C^{(N)\underline{\ell}\underline{m} }_{\ell m ,\ell' m'}$ with respect to $\mu=\frac{1}{N+1}$, one finds that the even powers of the series vanish. In fact, one can check the following identities:
\begin{align*}
S(k,L,-\mu)&=S(L-k,L,\mu)\\
R(\ell ,\ell' ,\underline{\ell},k)&=R(\ell ,\ell' ,\underline{\ell},L-k).
\end{align*}
These imply, relabelling $k$ with $L-k$, that:
\begin{align*}
 \sum_{k=\max\lbrace\ell,\ell',\underline{\ell}\rbrace}^{min\lbrace\ell+\ell',\ell'+\underline{\ell},\ell+\underline{\ell}\rbrace}\dfrac{(-1)^k S(k,L,\mu)}{R(\ell ,\ell' ,\underline{\ell},k)}&= \sum_{k=\max\lbrace\ell,\ell',\underline{\ell}\rbrace}^{min\lbrace\ell+\ell',\ell'+\underline{\ell},\ell+\underline{\ell}\rbrace}\dfrac{(-1)^{L-k} S(k,L,-\mu)}{R(\ell ,\ell' ,\underline{\ell},k)}\\
&=(-1)^{L} \sum_{k=\max\lbrace\ell,\ell',\underline{\ell}\rbrace}^{min\lbrace\ell+\ell',\ell'+\underline{\ell},\ell+\underline{\ell}\rbrace}\dfrac{(-1)^k S(k,L,-\mu)}{R(\ell ,\ell' ,\underline{\ell},k)}
\end{align*}
and so for even powers of $\mu$ only even $L$ terms survive but because of the coefficient $(1-(-1)^{l+\ell' +\underline{\ell}})$ in $C^{(N)\underline{\ell}\underline{m} }_{\ell m ,\ell' m'}$, these can be ignored.
Finally, since the term 
\[ \prod_{p_1=1}^{\ell}\prod_{p_2=1}^{\ell'}\prod_{p_3=1}^{\underline{\ell}}\Big(1-\Big(\frac{p_i}{N+1}\Big)^2\Big)^{-1/2}=1+\mathcal{O}\left(\frac{1}{(N+1)^2}\right),\]
the calculations above imply that the linear convergence proved in \cite{RiSt2014} is actually quadratic for $\ell ,\ell' ,\underline{\ell}\ll N$, i.e. for $\ell ,\ell' ,\underline{\ell}$ fixed while $N\rightarrow\infty$:
\[ 
C^{(N)\underline{\ell}\underline{m} }_{\ell m ,\ell' m'}=C^{\underline{\ell}\underline{m} }_{\ell m ,\ell' m'}+\mathcal{O}\left(\frac{1}{(N+1)^2}\right).
\]

\begin{lemma}[$C^{\underline{\ell}\underline{m} }_{\ell m ,\ell' m'}$ bounds]\label{lem:str_const_SH}
There exists a constant $C>0$ such that the structure constants of the spherical harmonics in the usual basis satisfy the following bound:
\[
|C^{\underline{\ell}\underline{m} }_{\ell m ,\ell' m'}|\leq C \min\lbrace \ell \ell' ,\ell \underline{\ell},\ell' \underline{\ell}\rbrace,
\]
for any $\ell ,\ell' ,\underline{\ell}=1,2,...$, satisfying the triangular inequality.
\end{lemma}
\proof
We have seen that the structure constants $C_{\ell m  \ell' m'}^{\underline{\ell}}$ can be written in the following way
\begin{align*}
C_{\ell m  \ell' m'}^{\underline{\ell} \underline{m}}=(1-(-1)^{\ell +\ell' +\underline{\ell}})(-1)^{\underline{m}+1}&\sqrt{2\ell +1}\sqrt{2\ell' +1}\sqrt{2\underline{\ell}+1}\cdot \\ 
&\cdot 
\left( \begin{array}{ccc}
\ell  & \ell'  & \underline{\ell} \\
m  & m' & -\underline{m}  
\end{array} \right)P(\ell ,\ell' ,\underline{\ell}).
\end{align*}

\textbf{Step 1.} Let's first focus on $P(\ell ,\ell' ,\underline{\ell})$. Let's first rewrite it in terms of $L,L_1 =L-2\ell ,L_2 =L-2\ell' ,L_3=L-2\underline{\ell}$:
\begin{align*}
P(\ell ,\ell' ,\underline{\ell})=(-1)^{(L_3+1)/2}\sqrt{\dfrac{L_3!L_2!L_1!}{(L+1)!}}(L+1)\dfrac{((L-1)/2)!}{((L_1-1)/2)!((L_2-1)/2)!((L_3-1)/2)!}
\end{align*}
Using the Stirling approximation of the factorial we get:
\begin{align*}
&P(\ell ,\ell' ,\underline{\ell})
\approx\sqrt[4]{\dfrac{L_1 L_2 L_3}{L+1}}\dfrac{e^{(L+1)/2}}{e^{L_3/2}e^{L_2/2}e^{L_1/2}}\dfrac{L_1^{L_1/2}L_2^{L_2/2}L_3^{L_3/2}}{(L+1)^{(L+1)/2}}(L+1)\sqrt{\dfrac{L-1}{(L_1-1)(L_2-1)(L_3-1)}}\cdot\\
&\cdot\dfrac{e^{(L_3-1)/2}e^{(L_2-1)/2}e^{(L_1 -1)/2}}{e^{(L-1)/2}}\dfrac{((L-1)/2)^{(L-1)/2}}{((L_1-1)/2)^{(L_1-1)/2}((L_2-1)/2)^{(L_2-1)/2}((L_3-1)/2)^{(L_3-1)/2}}
\\
&
\\
\approx&\sqrt{\dfrac{(L-1)(L_1 L_2 L_3)^{1/2}}{(L+1)^{1/2}(L_1-1)(L_2-1)(L_3-1)}}L_1^{1/2}L_2^{1/2}L_3^{1/2}\dfrac{L_1^{(L_1-1)/2}L_2^{(L_2-1)/2}L_3^{(L_3-1)/2}}{(L+1)^{(L-1)/2}}\cdot\\
&\cdot\dfrac{(L-1))^{(L-1)/2}}{(L_1-1)^{(L_1-1)/2}(L_2-1)^{(L_2-1)/2}(L_3-1)^{(L_3-1)/2}}\dfrac{(1/2)^{(L-1)/2}}{(1/2)^{(L_1-1)/2}(1/2)^{(L_2-1)/2}(1/2)^{(L_3-1)/2}}\\
\\
&
\\
\approx&\sqrt[4]{L L_1 L_2  L_3}\\
\end{align*}
where we have repeatedly used the equality: $L_1 +L_2 +L_3=L$. From this, using the definition of the $L_i$ and the fact that the $\ell,\ell',\underline{\ell}$ satisfy the triangular inequality, it is straightforward to check that:
\[
P(\ell ,\ell' ,\underline{\ell})\leq C \min\lbrace \sqrt{\ell }\sqrt{\ell' },\sqrt{\ell }\sqrt{\underline{\ell}},\sqrt{\ell' }\sqrt{\underline{\ell}}\rbrace.
\]

%We have that\todo[author=MV]{Is the exponent 3/4 or 1?}:
%\[1/2L^{1/2}\leq P(\ell ,\ell' ,\underline{\ell})\leq 1/3L^{3/4}.\] 

\textbf{Step 2.} For any $\ell^*\in\lbrace \ell,\ell', \underline{\ell}\rbrace$, we have (see \cite{RiSt2014}):
\[
\mid\sqrt{2\ell^*+1}\left( \begin{array}{ccc}
\ell  & \ell'  & \underline{\ell} \\
m  & m' & -\underline{m}  
\end{array}\right)\mid\leq 1
\]

\textbf{Step 3.} Finally, using the results in Step 1 and Step 2, we get:
\[
|C^{\underline{\ell}\underline{m} }_{\ell m ,\ell' m'}|\leq C \min\lbrace \ell \ell' ,\ell \underline{\ell},\ell' \underline{\ell}\rbrace,
\]
for some constant $C>0$.
\endproof

\begin{lemma}[$C^{(N) \underline{\ell}\underline{m} }_{\ell m ,\ell' m'}$ bounds]\label{lem:str_const_sun}
The structure constants $C^{(N) \underline{\ell}\underline{m} }_{\ell m ,\ell' m'}$ satisfy the following bounds. There exists some constant $C>1$ such that:
\begin{enumerate}
\item $C^{(N) \underline{\ell}\underline{m} }_{\ell m ,\ell' m'}\leq C C^{\underline{\ell}\underline{m} }_{\ell m ,\ell' m'}$, \hspace{1cm}for $\ell ,\ell' ,\underline{\ell}$ fixed, for $N\rightarrow\infty$;\\
\item $C^{(N) \underline{\ell}\underline{m} }_{\ell m ,\ell' m'}\leq CN$,\hspace{1cm} for only one index of $\lbrace\ell,\ell',\underline{\ell}\rbrace$ fixed, while the other two diverge, for $N\rightarrow\infty$;
\item $C^{(N) \underline{\ell}\underline{m} }_{\ell m ,\ell' m'}\leq C\sqrt{N}$,\hspace{1cm} for $\ell ,\ell' ,\underline{\ell}\rightarrow\infty$, for $N\rightarrow\infty$;
\end{enumerate}
for any $\ell ,\ell' ,\underline{\ell}=1,2,...$, satisfying the triangular inequalities.
\end{lemma}
\proof
\begin{enumerate}
\item We have the classical result \cite{RiSt2014}:
\[
C^{(N) \underline{\ell}\underline{m} }_{\ell m ,\ell' m'}=C^{\underline{\ell}\underline{m} }_{\ell m ,\ell' m'}+\mathcal{O}\left(\frac{1}{(N+1)^2}\right).
\]
Moreover, $C^{\underline{\ell}\underline{m} }_{\ell m ,\ell' m'}=0$ if and only if the 3j-symbol factor is zero or the triad $\ell ,\ell' ,\underline{\ell}$ does not respect the triangular inequalities. 
Therefore, if $C^{\underline{\ell}\underline{m} }_{\ell m ,\ell' m'}=0$ then $C^{(N) \underline{\ell}\underline{m} }_{\ell m ,\ell' m'}=0$. 
Hence, we can write 
\[
C^{(N) \underline{\ell}\underline{m} }_{\ell m ,\ell' m'}=C(\ell ,\ell' ,\underline{\ell},N)C^{\underline{\ell}\underline{m} }_{\ell m ,\ell' m'}
\]
where $C(\ell ,\ell' ,\underline{\ell},N)\rightarrow 1$, for $N\rightarrow\infty$. Therefore, for $N$ sufficiently large we find $C>1$ for which the thesis is valid.
\item Let us fix $\ell $ and let $\ell' ,\underline{\ell}$ going to infinity for $N\rightarrow\infty$. 
Using the the Edmonds asymptotic formula for the 6j-symbols \cite{Fl1998}:
\[
\left\lbrace\begin{array}{ccc}
\ell  & \ell'  & \underline{\ell} \\
N/2 & N/2 & N/2  
\end{array}\right\rbrace\leq \frac{C}{\sqrt{(2\ell' +1)(N+1)}}+\mathcal{O}(1/N^2)
\]
and the fact that
\[
\mid\sqrt{2\underline{\ell}+1}\left( \begin{array}{ccc}
\ell  & \ell'  & \underline{\ell} \\
m  & m' & -\underline{m}  
\end{array}\right)\mid\leq 1
\] find $C>0$ such that:
\[
C^{(N) \underline{\ell}\underline{m} }_{\ell m ,\ell' m'}\leq CN.
\]
Moreover, by the permutation properties of the 3j and 6j symbols, we obtain the same result permuting the three indexes $\ell ,\ell' ,\underline{\ell}$.
\item When all the coefficients of the triad $\ell ,\ell' ,\underline{\ell}$ grow simultaneously, for $N\rightarrow\infty$, we can use the Ponzano-Regge formula (see \cite{Gu2008}):
\[
\left\lbrace\begin{array}{ccc}
\ell  & \ell'  & \underline{\ell} \\
N/2 & N/2 & N/2  
\end{array}\right\rbrace\leq \frac{C}{\sqrt{N^3(2\ell +1)(2\ell' +1)(2\underline{\ell}+1)}}+\mathcal{O}(1/N^2)
\]
Let $\ell \sim N^{\alpha_1},\ell' \sim N^{\alpha_2},\underline{\ell}\sim N^{\alpha_3}$, for $0<\alpha_1,\alpha_2,\alpha_3<1$. Then, there exists a constant $C$ such that
\[
C^{(N) \underline{\ell}\underline{m} }_{\ell m ,\ell' m'}\leq C\sqrt{2\underline{\ell}+1}\leq C\sqrt{N}.
\]
\end{enumerate}
\endproof
We can now derive the core result in the consistency proof.
\begin{proposition}\label{prop:str_const_diff}
There exists a constant $C$ such that, for any $N$ odd and any set of admissible indexes $\ell ,m ,\ell' ,m',\underline{\ell},\underline{m}\leq N$, it holds:
\begin{enumerate}
\item $|C^{(N)\underline{\ell}\underline{m}}_{\ell m ,\ell' m'}-C^{\underline{\ell}\underline{m}}_{\ell m ,\ell' m'}|\leq
C\min\lbrace \ell \ell' ,\ell \underline{\ell},\ell' \underline{\ell}\rbrace$,\\
\item $N^2|C^{(N)\underline{\ell}\underline{m}}_{\ell m ,\ell' m'}-C^{\underline{\ell}\underline{m}}_{\ell m ,\ell' m'}|\leq
C\max\lbrace \ell ^2,\ell'^2,\underline{\ell}^2\rbrace\cdot\min\lbrace \ell \ell' ,\ell \underline{\ell},\ell' \underline{\ell}\rbrace$.
\end{enumerate}
\end{proposition}
\proof
\begin{enumerate}
\item 
Using Lemma~\ref{lem:str_const_SH} and Lemma~\ref{lem:str_const_sun}, we have that
\begin{align*}
|C^{(N)\underline{\ell}\underline{m}}_{\ell m ,\ell' m'}-C^{\underline{\ell}\underline{m}}_{\ell m ,\ell' m'}|&\leq|C^{(N)\underline{\ell}\underline{m}}_{\ell m ,\ell' m'}|+|C^{\underline{\ell}\underline{m}}_{\ell m ,\ell' m'}|\\
&\leq C(\max\lbrace N,\min\lbrace \ell \ell' ,\ell \underline{\ell},\ell' \underline{\ell}\rbrace\rbrace + \min\lbrace \ell \ell' ,\ell \underline{\ell},\ell' \underline{\ell}\rbrace).
\end{align*}
Since, for any set of admissible indexes, $|C^{(N)\underline{\ell}\underline{m}}_{\ell m ,\ell' m'}-C^{\underline{\ell}\underline{m}}_{\ell m ,\ell' m'}|\rightarrow 0$, for $N\rightarrow\infty$, and $\ell ,\ell' ,\underline{\ell}<N$, the bound can be replaced with $\max\lbrace \ell ,\ell' ,\underline{\ell},\min\lbrace \ell \ell' ,\ell \underline{\ell},\ell' \underline{\ell}\rbrace\rbrace + \min\lbrace \ell \ell' ,\ell \underline{\ell},\ell' \underline{\ell}\rbrace$, and so with $\min\lbrace \ell \ell' ,\ell \underline{\ell},\ell' \underline{\ell}\rbrace$.\\
\item 
We know that there exits a function $C(\ell ,\ell' ,\underline{\ell},N)$ such that 
\[\frac{C(\ell ,\ell' ,\underline{\ell},N)}{N^2}=|C^{(N)\underline{\ell}\underline{m}}_{\ell m ,\ell' m'}-C^{\underline{\ell}\underline{m}}_{\ell m ,\ell' m'}|\]
and $C(\ell ,\ell' ,\underline{\ell},N)\rightarrow \overline{C}(\ell ,\ell' ,\underline{\ell})\in \Rr$, for $N\rightarrow\infty$.
Moreover, 
\[C(\ell ,\ell' ,\underline{\ell},N)\leq CN^2\min\lbrace \ell \ell' ,\ell \underline{\ell},\ell' \underline{\ell}\rbrace.\] 
Hence, since $C(\ell ,\ell' ,\underline{\ell},N)\rightarrow \overline{C}(\ell ,\ell' ,\underline{\ell})\in \Rr$, for $N\rightarrow\infty$, $C(\ell ,\ell' ,\underline{\ell},N)$ can grow at most as $\max\lbrace \ell ^2,\ell'^2,\underline{\ell}^2\rbrace\cdot\min\lbrace \ell \ell' ,\ell \underline{\ell},\ell' \underline{\ell}\rbrace$.
\end{enumerate}
\endproof

\section{Structure constants estimates for the 2-torus}\label{sec:AppB}
In this section we show that the same calculations can be explicitly done also for the Zeitlin's mothe on the 2-torus (see \cite{ze1}).
Let $\omega(x,t) = \sum_{\kk\in\Zz^2_0}\omega_{\kk}(t)e^{i\kk\cdot x}$ be the vorticity field on $\Tt^2$. 
From now on, all the sums are taken excluding the index $0$.
Then, for each $\nn\in\Zz^2_0$, the equations of motion of $\omega_\nn$ are:
\begin{equation}
\dot{\omega}_\nn = B_\nn(\omega):= \sum_{k_1,k_2=-\infty}^\infty \dfrac{\nn\times\kk}{|\kk|^2}\omega_{\nn-\kk}\omega_{\kk},
\end{equation}
where $\nn\times\kk=n_2k_1-k_2n_1$.
Let $W^N(t) = \sum_{k_1,k_2=-(N-1)/2}^{(N-1)/2}\omega_{\kk}(t)T^N_\kk$ be its projection in $\su(N)$.
Then, for each $\nn\in\Zz^2_0$ such that $|n_1|,|n_2|\leq(N-1)/2$, the equations of motion of $\omega_\nn$ are:
\begin{equation}
\dot{\omega}_\nn = B^N_\nn(W^N):= \sum_{k_1,k_2=-(N-1)/2}^{(N-1)/2}\frac{N}{2\pi}\frac{\sin\left(\dfrac{2\pi}{N} \nn\times\kk\right)}{|\kk|^2}\omega_{\nn-\kk}\omega_{\kk},
\end{equation}
where the indices on the $\omega_\nn$ are taken $\mbox{mod} \ N$.

Let us introduce the reminder $r^N(\omega)=B^N(\Pi_N\omega)-\Pi_N B(\iota_N\circ\Pi_N\omega)$, where $\Pi_N:L^2(\Ss^2)\rightarrow\su(N)$ is the orthogonal projection and $\iota_N:\su(N)\rightarrow L^2(\Ss^2)$ is the inclusion such that $\iota_N\circ\Pi_N$ correspond to the standard truncation of the Fourier series.
In components, we have that:
\begin{equation}
r^N(\omega) = \sum_{k_1,k_2=-(N-1)/2}^{(N-1)/2}\left[\frac{N}{2\pi}\frac{\sin\left(\dfrac{2\pi}{N} \nn\times\kk\right)}{|\kk|^2}- \dfrac{\nn\times\kk}{|\kk|^2}\right]\omega_{\nn-\kk}\omega_{\kk}.
\end{equation}

Then, let $\mu_N(dW) = \frac{1}{Z_{d_N}}e^{-1/2\|W\|^2}dW$ be the Gaussian measure on $\su(N)$ and let $W^N:=\iota_N(\omega):\Omega_N\rightarrow\su(N)$ be distributed as $\mu_N$. 
We want to estimate $\Ee^{\mu_N}[\|r^N(\omega)\|^2_{-s}]$, for some $s>0$.
Let us call:
\[
C^N_{\nn,\kk}:=\frac{N}{2\pi}\sin\left(\dfrac{2\pi}{N} \nn\times\kk\right)- \nn\times\kk.
\]
\begin{equation}
\begin{array}{ll}
\Ee^{\mu_N}[\|r^N(\omega)\|^2_{-s}] &= \Ee^{\mu_N}\left[\sum_{n_1,n_2=-(N-1)/2}^{(N-1)/2}\frac{1}{|\nn|^{2s}}\left|\sum_{k_1,k_2=-(N-1)/2}^{(N-1)/2}\dfrac{C^N_{\nn,\kk}}{|\kk|^2}\omega_{\nn-\kk}\omega_{\kk}\right|^2\right]\\
&=\Ee^{\mu_N}\left[\sum_{n_1,n_2=-(N-1)/2}^{(N-1)/2}\frac{1}{|\nn|^{2s}}\sum_{\kk,\kk'}\dfrac{C^N_{\nn,\kk}}{|\kk|^2}\dfrac{C^N_{\nn,\kk'}}{|\kk'|^2}\omega_{\nn-\kk}\omega_{\kk}\overline{\omega}_{\nn-\kk'}\overline{\omega}_{\kk'}\right]\\
&=\sum_{n_1,n_2=-(N-1)/2}^{(N-1)/2}\frac{1}{|\nn|^{2s}}\sum_{\kk,\kk'}\dfrac{C^N_{\nn,\kk}}{|\kk|^2}\dfrac{C^N_{\nn,\kk'}}{|\kk'|^2}\Ee^{\mu_N}\left[\omega_{\nn-\kk}\omega_{\kk}\overline{\omega}_{\nn-\kk'}\overline{\omega}_{\kk'}\right].\\
\end{array}
\end{equation}
By the Isserlis-Wick formula:
\[\Ee^{\mu_N}\left[\omega_{\nn-\kk}\omega_{\kk}\overline{\omega}_{\nn-\kk'}\overline{\omega}_{\kk'}\right]=\delta_\kk^{\kk'} + \delta_{\nn-\kk}^{\kk'}.\]
Hence, using the fact that $C^N_{\nn,\kk}=-C^N_{\nn,\nn-\kk}$, we get:
\begin{equation}
\begin{array}{ll}
\Ee^{\mu_N}[\|r^N(\omega)\|^2_{-s}] &= \sum_{n_1,n_2=-(N-1)/2}^{(N-1)/2}\frac{1}{|\nn|^{2s}}\left(\sum_{\kk}\dfrac{(C^N_{\nn,\kk})^2}{|\kk|^4}-\sum_{\kk}\dfrac{(C^N_{\nn,\kk})^2}{|\kk|^2|\nn-\kk|^2}\right)\\
&= \sum_{n_1,n_2=-(N-1)/2}^{(N-1)/2}\frac{1}{|\nn|^{2s}}\left(\sum_{\kk}(C^N_{\nn,\kk})^2\dfrac{|\nn-\kk|^2-|\kk|^2}{|\kk|^4|\nn-\kk|^2}\right)\\
&\leq \sum_{n_1,n_2=-(N-1)/2}^{(N-1)/2}\frac{1}{|\nn|^{2s}}\left(\sum_{\kk}(C^N_{\nn,\kk})^2\dfrac{|\nn|(|\nn-\kk|+|\kk|)}{|\kk|^4|\nn-\kk|^2}\right).
\end{array}
\end{equation}
Now, using the fact that $|\sin x - x|\leq C |x|^3$, we get that: 
\[|C^N_{\nn,\kk}|\leq C N|\frac{1}{N} \nn\times\kk|^3=C \frac{1}{N^2}| \nn\times\kk|^3.\]
Therefore, we have that:
\begin{equation}
\begin{array}{ll}
\Ee^{\mu_N}[\|r^N(\omega)\|^2_{-s}] &\leq C\sum_{n_1,n_2=-(N-1)/2}^{(N-1)/2}\frac{1}{|\nn|^{2s-1}}\left(\sum_{\kk}\dfrac{| \nn\times\kk|^6(|\nn-\kk|+|\kk|)}{N^4|\kk|^4|\nn-\kk|^2}\right)\\
&\leq \dfrac{C}{N^4}\sum_{n_1,n_2=-(N-1)/2}^{(N-1)/2}\frac{1}{|\nn|^{2s-7}}\sum_{\kk}\dfrac{|\kk|^2}{|\nn-\kk|}+\dfrac{|\kk|^3}{|\nn-\kk|^2}\\
&\leq \dfrac{C}{N^4}\sum_{n_1,n_2=-(N-1)/2}^{(N-1)/2}\frac{1}{|\nn|^{2s-7}}\sum_{\kk}\dfrac{|\nn-\kk|^2+|\nn|^2}{|\nn-\kk|}+\dfrac{|\nn-\kk|^3+|\nn|^3}{|\nn-\kk|^2}\\
&\leq \dfrac{C}{N}\sum_{n_1,n_2=-(N-1)/2}^{(N-1)/2}\frac{1}{|\nn|^{2s-7}}+\dfrac{C}{N^4}\sum_{n_1,n_2=-(N-1)/2}^{(N-1)/2}\frac{|\nn|^2N\log N + |\nn|^3\log N}{|\nn|^{2s-7}}\\
&\leq C\left(\dfrac{N^{9-2s}}{N}+\dfrac{N^{12-2s}\log N}{N^4}\right)\\
\end{array}
\end{equation}
which goes to $0$ for $N\rightarrow\infty$ for $s>9/2$ .

\bibliographystyle{plain}
\bibliography{biblio.bib}

\end{document}